\documentclass[12pt,a4paper]{article}
\usepackage{amsfonts,amssymb,amsmath,amsthm}
\usepackage[T1]{fontenc}
\usepackage[english]{babel} 
\usepackage[utf8]{inputenc}
\usepackage{a4wide}

\begin{document}

\newtheorem{Theorem}{Theorem}[section]
\newtheorem{Corollary}[Theorem]{Corollary}
\newtheorem{Proposition}[Theorem]{Proposition}
\newtheorem{Lemma}[Theorem]{Lemma}
\newtheorem{Assumptions}[Theorem]{Assumptions}
\newtheorem{definition}{Definition}[section]

\theoremstyle{definition}
\newtheorem{Remark}[Theorem]{Remark}

\title{\bf Stochastic evolution equations with Wick-polynomial nonlinearities}

\author{
Tijana~Levajkovi\'c\thanks{University of Belgrade, Serbia, t.levajkovic@sf.bg.ac.rs} ,
Stevan~Pilipovi\'c\thanks{Faculty of Sciences, University of Novi Sad, Serbia, stevan.pilipovic@dmi.uns.ac.rs} ,
Dora~Sele\v si\thanks{Faculty of Sciences, University of Novi Sad, Serbia, dora.selesi@dmi.uns.ac.rs} ,
Milica \v Zigi\'c\thanks{Faculty of Sciences, University of Novi Sad, Serbia, milica.zigic@dmi.uns.ac.rs}
	} 
\date{}
\maketitle


\bigskip

\begin{abstract}
We study nonlinear parabolic stochastic partial differential equations with Wick-power and Wick-polynomial type nonlinearities set in the
framework of white noise analysis. These equations include the stochastic Fujita equation, the stochastic Fisher-KPP equation and the stochastic FitzHugh-Nagumo equation among many others. By implementing the theory of $C_0-$semigroups and evolution systems into the chaos expansion theory in infinite dimensional spaces, we  prove existence and uniqueness of solutions for this class of SPDEs. 
In particular, we also treat the linear nonautonomous case and provide several applications featured as stochastic reaction-diffusion equations that arise in biology, medicine and physics.
\end{abstract}


\section{Introduction}
We study stochastic nonlinear evolution equations of the form  
\begin{align}\label{PNLJ}
u_t (t, \omega)&= \mathbf A \, u(t, \omega) + \sum_{k=0}^n a_k u^{\lozenge k}(t, \omega)+f(t,\omega), \quad t\in (0,T]\\\nonumber
u(0,\omega) &= u^0(\omega), \quad \omega\in\Omega,
\end{align} 
where $u(t,\omega)$ is an $X-$valued generalized stochastic process; $X$ is a certain  Banach algebra and $\mathbf A$ corresponds to a densely defined infinitesimal  generator of a $C_0-$semigroup. 
The nonlinear part is given in terms of Wick-powers $u^{\lozenge n} = u^{\lozenge n-1} \lozenge u=u\lozenge \dots \lozenge u,\;n\in \mathbb{N}$, where $\lozenge$ denotes the Wick product.  The Wick product is involved due to the fact that we allow random terms to be present both in the initial condition $u_0$
and the driving force $f$. This leads to singular solutions that do not allow to use ordinary multiplication, but require a renormalization of the multiplication, which is done by introducing the Wick product into the equation. The Wick product is known to represent the highest order stochastic approximation of the ordinary product \cite{Mikulevicius}. 

In our previous paper \cite{Milica} we treated the case of linear stochastic parabolic equations with Wick-multipli\-ca\-ti\-ve noise which includes the case $n=1.$ The present paper is an extension of \cite{Milica} to nonlinear equations, where the nonlinearity is generated by a Wick-polynomial function leading to stochastic versions of Fujita-type equations $u_t = \mathbf A u + u^{\lozenge n}+f$, FitzHugh-Nagumo equations  $u_t = \mathbf A u + u^{\lozenge 2} - u^{\lozenge 3}+f$, Fisher-KPP equations $u_t = \mathbf A u + u - u^{\lozenge 2}+f$ and Chaffee-Infante equations  $u_t = \mathbf A u + u^{\lozenge 3} - u+f$. These equations have found ample applications in ecology, medicine, engineering and physics. For example, the FitzHugh-Nagumo equation is used to study electrical activity of neurons in neurophysiology by modeling the conduction of electric impulses down a nerve axon. The Fisher-Kolmogorov-Petrovsky-Piskunov equation provides a model for the spread of an epidemic in a population or for the distribution of an advantageous gene within a
population. Other applications in medicine involve the modeling of cellular reactions to the introduction of toxins, and the process of epidermal wound healing. In plasma physics it has been used to study  neutron flux in nuclear reactors, while in ecology it models flame propagation of fire outbreaks. Thus, the study of their stochastic versions, when some of the input factors is disturbed by an external noise factor and hence it becomes randomized, is of immense importance. For instance, a stochastic version of the FitzHugh-Nagumo equation has been studied in \cite{Albeverio} and \cite{Barbu}, while the stochastic Fisher-KPP equation has been studied in \cite{HuangLiu} and \cite{OksendalVZ}. 

 We implement the Wiener-It\^o chaos expansion method combined with the operator semigroup theory in order to prove the existence and the uniqueness of a solution for \eqref{PNLJ}. Using the chaos expansion method any SPDE can be transformed into a lower triangular infinite system of PDEs (also known as the propagator system) that can be solved recursively. Solving this system, one obtains the coefficients of the solution to \eqref{PNLJ}. In order to solve the propagator system, we exploit the intrinsic relationship between the Wick product and the Catalan numbers that was discovered in \cite{KL} where the authors considered the stochastic Burgers equation. We build upon these ideas in order to solve a general class of stochastic nonlinear equations \eqref{PNLJ}.

The plan of exposition is as follows: In the introductory section we recall upon basic notions of $C_0-$semigroups, evolution systems and white noise theory including chaos expansions of generalized stochastic processes. In Section 2, which represents the main part of the paper, we prove existence and uniqueness of the solution to \eqref{PNLJ} for the case when $a_0=a_1=\cdots=a_{n-1}=0$ and $a_n=1$. This normalization is made for technical simplicity to illustrate the method of solving and to put out in details all building blocks of the formulae involved. In Section 3 we treat the general case of \eqref{PNLJ} and provide some concrete examples.

\subsection{Evolution systems}

We fix the notation and recall some known facts about evolution systems (see \cite[Chapter 5]{Pazy}). Let $X$ be a Banach space. Let $\{A(t)\}_{t\in[s,T]}$ be a family of linear operators in $X$  such that $A(t):D(A(t))\subset X\to X,\;t\in[s,T]$ and let $f$ be an $X-$valued function $f:[s,T]\to X.$  Consider the initial value problem 
\begin{align}\label{EP}
\frac{d}{dt}u(t)&=A(t)u(t)+f(t),\quad 0\leq s<t\leq T,\\
u(s)&=x.\nonumber
\end{align}
If $u\in C([s,T],X)\cap C^1((s,T],X),$ $u(t)\in D(A(t))$ for all $t\in (s,T]$ and $u$ satisfies \eqref{EP}, then $u$ is a classical solution of \eqref{EP}.

A two parameter family of bounded linear operators $S(t,s),\;0\leq s\leq t\leq T$ on X is called an evolution system if the following two conditions are satisfied:
\begin{enumerate}
\item $S(s,s)=I$ and $S(t,r)S(r,s)=S(t,s),\quad 0\leq s\leq r\leq t\leq T$
\item $(t,s)\mapsto S(t,s)$ is strongly continuous for all $0\leq s\leq t\leq T.$
\end{enumerate}
Clearly, if $S(t,s)$ is an evolution system associated with the homogeneous evolution problem \eqref{EP}, i.e. if $f\equiv 0,$ then a classical solution of \eqref{EP} is given by
$
u(t)=S(t,s)x,\ t\in[s,T].
$

A family $\{A(t)\}_{t\in[s,T]}$ of infinitesimal generators of $C_0-$semigroups on $X$ is called stable if there exist constants $m\geq 1$ and $w\in\mathbb{R}$ (stability constants) such that $(w,\infty)\subseteq \rho(A(t)),\; t\in[s,T]$ and
$$\Big\|\prod_{j=1}^kR(\lambda:A(t_j))\Big\|\leq\frac{m}{(\lambda-w)^k},\quad \lambda>w,$$
 for every finite sequence $0\leq s\leq t_1\leq t_2\leq\dots\leq t_k\leq T,\;k=1,2,\dots.$

Let $\{A(t)\}_{t\in [s,T]}$ be a stable family of infinitesimal generators with stability constants $m$ and $w$. Let $B(t),\;t\in[s,T],$ be a family of bounded linear operators on $X$. If $\|B(t)\|\leq M,\;t\in [s, T],$ then $\{A(t)+B(t)\}_{t\in[s,T]}$ is a stable family of infinitesimal generators with stability constants $m$ and $w+Mm.$

Let $\{A(t)\}_{t\in [s,T]}$ be a stable family of infinitesimal generators of $C_0-$semigroups on $X$ such that the domain  $D(A(t))=D$ is independent of $t$ and for every $x\in D,$ $A(t)x$ is continuously differentiable in $X.$ If $f\in C^1([s,T],X)$ then for every $x\in D$ the evolution problem \eqref{EP} has a unique classical solution $u$ given by
$$u(t)=S(t,s)x+\int_s^t S(t,r)f(r)dr,\quad 0\leq s\leq t\leq T.$$
From the proof of \cite[Theorem 5.3, p. 147]{Pazy} one can obtain
\begin{align*}
\frac{d}{dt}u(t)= A(t)S(t,s)x+A(t)\int_s^t S(t,r)f(r)dr+f(t),\quad s<t\leq T.
\end{align*}
Since $t\mapsto A(t)$ is continuous in $B(D,X)$ and $(t,s)\mapsto S(t,s)$ is strongly continuous for all $0\leq s\leq t\leq T,$ we have additionally that the solution $u$ to \eqref{EP} exhibits the regularity property $u\in C^1([s,T],X)$ and $\frac{d}{dt}u(t)|_{t=s}=A(s)x+f(s).$
Recall that the evolution system $S(t,s)$ satisfies:
\begin{enumerate}
\item $\|S(t,s)\|\leq me^{w(t-s)},\ 0\leq s\leq t\leq T;$
\item $\frac{\partial^+}{\partial t}S(t,s)x\Big|_{t=s}=A(s)x,\  x\in D,\  0\leq s\leq T$ 
which implies that $\frac{\partial}{\partial t}S(t,s)x=A(t)S(t,s)x$ since $t\mapsto A(t)$ is continuous in $B(D,X)$;
\item $\frac{\partial}{\partial s}S(t,s)x=-S(t,s)A(s)x,\  x\in D,\  0\leq s\leq t\leq T;$
\item $S(t,s)D\subseteq D;$
\item $S(t,s)x$ is continuous in $D$ for all $0\leq s\leq t\leq T$ and $x\in D.$
\end{enumerate}

\begin{Remark}
Considering infinitezimal generators depending on $t,$ we follow the standard  approach of Yosida (cf. \cite{51}, \cite{23}). We refer to \cite{NZ} for a method
based on an equivalent  operator extension problem (see also references in \cite{NZ}). The chaos expansion approach, which is the essence of our paper,
requires the existence results for the propagator system i.e. for the coordinate-wise deterministic Cauchy problems. For this purpose we demonstrate the applications of the hyperbolic Cauchy problem given in \cite{Pazy}.
\end{Remark}

\subsection{Generalized stochastic processes}\label{0.2}

Denote by $(\Omega, \mathcal{F}, \mu) $ the Gaussian white noise
probability space $(S'(\mathbb{R}), \mathcal{B}, \mu), $ where
$\Omega=S'(\mathbb{R})$ denotes the space of tempered distributions,
$\mathcal{B}$ the Borel sigma-algebra generated by the weak topology
on $S'(\mathbb{R})$ and $\mu$ the Gaussian white noise measure corresponding to  the
  characteristic function
\begin{equation*}\label{BM theorem}
\int_{S'(\mathbb{R})} \,  e^{{i\langle\omega, \phi\rangle}}
d\mu(\omega) = \exp \left [-\frac{1}{2} \|\phi\|^2_{L^2(\mathbb{R})}\right], \quad \quad\phi\in  S(\mathbb{R}),
\end{equation*}
given by the Bochner-Minlos theorem.

We recall the notions related to $L^2(\Omega,\mu)$ (see \cite{HOUZ}).
The set of multi-indices $\mathcal I$ is $(\mathbb N_0^\mathbb N)_c$,
i.e. the set of sequences of non-negative integers which have only
finitely many nonzero components. Especially, we denote by $\mathbf 0=(0,0,0,\ldots)$ the zero multi-index with all entries equal to zero, the length of a multi-index is $|\alpha|=\sum_{i=1}^\infty\alpha_i$ for $\alpha=(\alpha_1,\alpha_2,\ldots)\in\mathcal I$ and $\alpha!=\prod_{i=1}^\infty\alpha_i!.$ We will use the convention that $\alpha-\beta$ is defined if $\alpha_n-\beta_n\geq 0$ for all $n\in\mathbb N$, i.e., if $\alpha-\beta\geq\mathbf 0.$

The
Wiener-It\^o theorem (sometimes also referred to as the
Cameron-Martin theorem) states that one can define an orthogonal
basis $\{H_\alpha\}_{\alpha\in\mathcal I}$ of $L^2(\Omega,\mu)$,
where $H_\alpha$ are constructed by  means of Hermite orthogonal
polynomials $h_n$ and Hermite functions $\xi_n$,
$$H_\alpha(\omega)=\prod_{n=1} ^\infty h_{\alpha_n}(\langle\omega,\xi_n\rangle),\quad \alpha=(\alpha_1,\alpha_2,\ldots, \alpha_n\ldots)\in\mathcal I,\quad \omega\in\Omega.$$
Then, every $F\in L^2(\Omega,\mu)$ can be represented via the so
called \emph{chaos expansion}
$$F(\omega)=\sum_{\alpha\in\mathcal I} f_\alpha H_\alpha(\omega), \quad \omega\in S'(\mathbb{R}),\quad\sum_{\alpha\in\mathcal I} |f_\alpha|^2\alpha!<\infty,\quad f_\alpha\in\mathbb{R},\quad\alpha\in\mathcal I.$$

Denote by $\varepsilon_k=(0,0,\ldots, 1, 0,0,\ldots),\;k\in \mathbb{N}$ the
multi-index with the entry 1 at the $k$th place. Denote by $\mathcal
H_1$ the subspace of $L^2(\Omega,\mu)$, spanned by the polynomials
$H_{\varepsilon_k}(\cdot)$, $k\in\mathbb N$. The subspace $\mathcal
H_1$ contains Gaussian stochastic processes, e.g. Brownian motion is
given by the chaos expansion $B(t,\omega) = \sum_{k=1}^\infty
\int_0^t \xi_k(s)ds\;H_{\varepsilon_k}(\omega).$

Denote by $\mathcal H_m$ the $m$th order chaos space, i.e. the
closure in $L^2(\Omega,\mu)$ of the linear subspace spanned by the orthogonal polynomials
$H_\alpha(\cdot)$ with $|\alpha|=m$, $m\in\mathbb N_0$. Then the
Wiener-It\^o chaos expansion states that
$L^2(\Omega,\mu)=\bigoplus_{m=0}^\infty \mathcal H_m$, where
$\mathcal H_0$ is the set of constants in $L^2(\Omega,\mu)$.


Changing the topology on $L^2(\Omega,\mu)$ to a weaker one, T. Hida
\cite{Hida} defined spaces of generalized random variables
containing the white noise as a weak derivative of the Brownian
motion. We refer to \cite{Hida}, \cite{HOUZ} for white noise
analysis.

Let $(2\mathbb N)^{\alpha}=\prod_{n=1}^\infty (2n)^{\alpha_n},\quad
\alpha=(\alpha_1,\alpha_2,\ldots, \alpha_n,\ldots)\in\mathcal I.$ We
will often use the fact that the series $\sum_{\alpha\in\mathcal I}(2\mathbb N)^{-p\alpha}$ converges  for $p>1$. Using the same technique as in \cite[Chapter 2]{HOUZ} one can define Banach spaces $(S)_{\rho,p}$ of test functions and their topological duals $(S)_{-\rho,-p}$ of stochastic distributions for all $\rho\geq 0$ and $p\geq 0.$

\begin{definition} The stochastic test function spaces are defined by
$$(S)_{\rho,p} =\{F=\sum_{\alpha\in\mathcal I}f_\alpha {H_\alpha}\in L^2(\Omega,\mu):\;  \|F\|^2_{(S)_{\rho,p}}= \sum_{\alpha\in\mathcal I}(\alpha!)^{1+\rho} |f_\alpha|^2(2\mathbb N)^{p\alpha}<\infty\},$$
for all $\rho\geq 0,\;p\geq 0.$\\
Their topological duals, the stochastic distribution spaces, are given by formal sums:
$$(S)_{-\rho,-p} =\{F=\sum_{\alpha\in\mathcal I}f_\alpha {H_\alpha}:\;  \|F\|^2_{(S)_{-\rho,-p}}= \sum_{\alpha\in\mathcal I}(\alpha!)^{1-\rho}|f_\alpha|^2(2\mathbb N)^{-p\alpha}<\infty\},$$
for all $\rho\geq 0,\;p\geq 0.$\\
The space of test random variables is $(S)_{\rho} =\bigcap_{p\geq 0}(S)_{\rho,p},\;\rho\geq 0$ endowed with the projective topology.\\
Its dual, the space of generalized random variables is $(S)_{-\rho}
=\bigcup_{p\geq 0}(S)_{-\rho,-p},\;\rho\geq 0$ endowed with the inductive
topology.
\end{definition}

The action of  $F=\sum_{\alpha\in\mathcal{I}}  b_\alpha H_\alpha\in (S)_{-\rho}$ onto $f=\sum_{\alpha\in\mathcal{I}}c_\alpha H_\alpha\in(S)_{\rho}$ is given by $\langle F,f\rangle=\sum_{\alpha\in\mathcal{I}}(b_\alpha,c_\alpha)\alpha!,$ where $(b_\alpha,c_\alpha)$ stands for the inner product in $\mathbb{R}.$ Thus, they form a Gelfand triplet
$$(S)_{\rho} \subseteq L^2(\Omega,\mu) \subseteq
(S)_{-\rho},\quad \rho\geq 0.$$ 

Clearly,  the spaces $(S)_{\rho,p}$ and $(S)_{-\rho,-p}$ are separable Hilbert spaces. Moreover, $(S)_{\rho}$ and $(S)_{-\rho}$ are nuclear spaces.

For $\rho=0$ we obtain the space of Hida stochastic distributions $(S)_{-0}$ and for $\rho=1$ the Kondratiev space of generalized random variables $(S)_{-1}.$ It holds that
$$(S)_1\hookrightarrow (S)_{0} \hookrightarrow L^2(\Omega,\mu) \hookrightarrow
(S)_{-0}\hookrightarrow (S)_{-1},$$
where $\hookrightarrow$ denotes dense inclusions. Usually the values of $\rho$ are restricted to $\rho\in[0,1]$ in order to establish the $S-$transform (see \cite{Hida}, \cite{HOUZ}) when solving SPDEs, but in our case values $\rho>1$ may be considered as well.

The time-derivative of  the Brownian motion $B(t,\omega)=\sum_{k=1}^\infty \int_{0}^t\xi_k(s)ds\;H_{\varepsilon_k}(\omega)$ exists in a generalized
sense and belongs to the Kondratiev space $(S)_{-1,-p}$ for
$p\geq\frac5{12}$. We refer it as the \emph{white noise}
and its formal expansion is given by
 $W(t,\omega) =
\sum_{k=1}^\infty \xi_k(t)H_{\varepsilon_k}(\omega).$

We extended in \cite{GRPW} the definition of stochastic processes to processes with the chaos expansion form
$U(t,\omega)=\sum_{\alpha\in\mathcal I}u_\alpha(t)
{H_\alpha}(\omega)$, where the coefficients $u_\alpha$ are elements
of some Banach space of functions $X$. We say that $U$ is an \emph{$X$-valued generalized stochastic process}, i.e. $U(t,\omega)\in X\otimes (S)_{-\rho}$ if there exists $p\geq 0$ such that $\|U\|_{X\otimes(S)_{-\rho,-p}}^2=\sum_{\alpha\in\mathcal I}(\alpha!)^{1-\rho}\|u_\alpha\|_X^2(2\mathbb N)^{-p\alpha}<\infty$.

For example, let $X=C^k[0,T]$, $k\in\mathbb N$. We have proved in \cite{ps} that the differentiation of a stochastic process can be carried
out componentwise in the chaos expansion, i.e. due to the fact that
$(S)_{-\rho}$ is a nuclear space it holds that
$C^k([0,T],(S)_{-\rho})=C^k[0,T]\hat\otimes(S)_{-\rho}$ where $\hat{\otimes}$ denotes the completion of the tensor product which is the same for the $\varepsilon-$completion  and $\pi-$completion. In the sequel, we will use the notation $\otimes$ instead of $\hat\otimes$.  Hence $C^k[0,T]\otimes (S)_{-\rho,-p}$ and $C^k[0,T]\otimes (S)_{\rho,p}$ denote subspaces of the corresponding completions. We keep the same notation when $C^k[0,T]$ is replaced by another Banach space. This means that a
stochastic process $U(t,\omega)$ is $k$ times continuously
differentiable if and only if all of its coefficients $u_\alpha(t)$,
$\alpha\in\mathcal I$ are in $C^k[0,T]$.

The same holds for Banach space valued stochastic processes i.e.
elements of  $C^k([0,T],X)\otimes(S)_{-\rho}$, where $X$ is an
arbitrary Banach space. It holds that
$$C^k([0,T],X\otimes
(S)_{-\rho})=C^k([0,T],X)\otimes(S)_{-\rho}=\bigcup_{p\geq 0}C^k([0,T],X)\otimes
(S)_{-\rho,-p}.$$

In addition, if $X$ is a Banach algebra, then the \emph{Wick product} of the stochastic
processes $F=\sum_{\alpha\in\mathcal I}f_\alpha H_\alpha$ and $G=\sum_{\beta\in\mathcal I}g_\beta H_\beta\in X\otimes(S)_{-\rho,-p}$
is given by  $$F\lozenge G = \sum_{\gamma\in\mathcal
	I}\sum_{\alpha+\beta=\gamma}f_\alpha g_\beta H_\gamma =
\sum_{\alpha\in\mathcal I}\sum_{\beta\leq \alpha} f_\beta
g_{\alpha-\beta} H_\alpha,$$ and $F\lozenge G\in X\otimes(S)_{-\rho,-(p+k)}$ for all $k>1$ (see \cite{HOUZ}). The $n$th Wick power is defined by
$F^{\lozenge n}=F^{\lozenge (n-1)}\lozenge F$, $F^{\lozenge 0}=1$.
Note that
$H_{n\varepsilon_k}=H_{\varepsilon_k}^{\lozenge n}$ for $n\in\mathbb
N_0$, $k\in\mathbb N$. Throughout the paper we will assume that $X$ is a Banach algebra.

\section{Stochastic nonlinear evolution equation of Fujita-type} 

First we consider the equation \eqref{PNLJ}, with $a_0=a_1=\cdots=a_{n-1}=0$ and $a_n=1$, i.e. the equation:

\begin{align}\label{NLJ}
u_t (t, \omega)&= \mathbf A \, u(t, \omega) + u^{\lozenge n}(t, \omega)+f(t,\omega), \quad t\in (0,T]\\\nonumber
u(0,\omega) &= u^0(\omega), \quad \omega\in\Omega.
\end{align} 

 Let $\mathbf A: \mathbb{D}\subset X\otimes (S)_{-1} \to X\otimes (S)_{-1}$ be a coordinatewise operator that corresponds to a family of deterministic operators $A_\alpha: \, D_\alpha\subset X \to X$, $\alpha\in \mathcal{I}$ 
 $$\mathbf{A}u(t,\omega)=\mathbf{A}\left(\sum_{\alpha\in \mathcal I} u_\alpha(t) \, H_\alpha(\omega)\right)=\sum_{\alpha\in \mathcal I} A_\alpha u_\alpha(t) \,\,  H_\alpha(\omega),\quad u\in \mathbb{D},$$
 (see \cite[Section 2]{Milica}). We are looking for a solution of \eqref{NLJ} as an $X$-valued stochastic process $u(t)\in X\otimes(S)_{-1},\;t\in[0,T]$ represented in the form 
 	\begin{equation}
 	\label{proces}
 	u(t,\omega)= \sum_{\alpha\in \mathcal I} u_\alpha(t) \,\,  H_\alpha(\omega),\quad t\in[0,T],\quad \omega\in \Omega.
 	\end{equation}
 The chaos expansion representation of the Wick-square is given by
 \begin{align}\label{WP2}
 u^{\lozenge 2}(t,\omega)&= \sum_{\alpha \in \mathcal I} \Big( \sum_{\gamma\leq \alpha } \, u_\gamma(t) \,\, u_{\alpha-\gamma} (t)\Big) \, H_\alpha(\omega)\\\nonumber
 &=u^2_{\mathbf{0}}(t)\,H_{\mathbf{0}}(\omega)+\sum_{|\alpha|>0} \Big(2u_{\mathbf{0}}(t)\,u_\alpha(t)+ \sum_{0<\gamma< \alpha } \, u_\gamma(t) \,\, u_{\alpha-\gamma} (t)\Big) \, H_\alpha(\omega),
 \end{align}
 where $t\in[0,T],\;\omega\in\Omega.$
Let $u^{\lozenge m}_\gamma(t)$, $\gamma\in \mathcal I$, $m\in \mathbb N$ denote the coefficients of the chaos expansion of the $m$th Wick power, i.e.  $u^{\lozenge m}(t,\omega) = \sum_{\gamma\in \mathcal I} u^{\lozenge m}_\gamma(t) H_\gamma (\omega)$, for $m\in \mathbb N$. Then, for arbitrary $n\in \mathbb{N},$ it can be shown that the $n$th Wick-power is given by
 \begin{align*}\label{WPn}
 &u^{\lozenge n}(t,\omega)=u^{\lozenge n-1}(t,\omega)\lozenge u(t,\omega) =\sum_{\alpha \in \mathcal I} \Big( \sum_{\gamma\leq \alpha } \, u^{\lozenge n-1}_\gamma(t) \,\, u_{\alpha-\gamma} (t)\Big) \, H_\alpha(\omega)\\\nonumber
 &=u^n_{\mathbf{0}}(t)\,H_{\mathbf{0}}(\omega)+\sum_{|\alpha|>0} \Bigg(\binom{n}{1}u_{\mathbf{0}}^{n-1}(t)\,u_\alpha(t)+\binom{n}{2}u_{\mathbf{0}}^{n-2} \sum_{0<\gamma_1< \alpha } \, u_{\alpha-\gamma_1}(t) \,\, u_{\gamma_1} (t)\\\nonumber
 &+ \binom{n}{3}u_{\mathbf{0}}^{n-3} \sum_{0<\gamma_1< \alpha } \sum_{0<\gamma_2< \gamma_1 }\, u_{\alpha-\gamma_1}(t) \,\, u_{\gamma_1-\gamma_2} (t) u_{\gamma_2}(t)+\dots+\\\nonumber
 &+ \binom{n}{n}\sum_{0<\gamma_1< \alpha } \sum_{0<\gamma_2< \gamma_1 }\dots \sum_{0<\gamma_{n-1}< \gamma_{n-2} }\, u_{\alpha-\gamma_1}(t) \,\, u_{\gamma_1-\gamma_2} (t)\dots  u_{\gamma_{n-2}-\gamma_{n-1}}(t)u_{\gamma_{n-1}}(t)\Bigg) \, H_\alpha(\omega)\\\nonumber
 &= u^n_{\mathbf{0}}(t)\,H_{\mathbf{0}}(\omega)+\sum_{|\alpha|>0} \Bigg(n\,u_{\mathbf{0}}^{n-1}(t)\,u_\alpha(t)+r_{\alpha,n} (t)\Bigg) \, H_\alpha(\omega),
 \end{align*}
where $t\in[0,T],\;\omega\in\Omega.$ The functions $r_{\alpha,n}(t),\;t\in[0,T],\;\alpha\in\mathcal{I},\;n>1$ contain only the coordinate functions $u_\beta,\;\beta< \alpha.$ Moreover, we recall that the Wick power $u^{\lozenge n}$ of a stochastic process $u\in X\otimes (S)_{-1, -p}$ is an element of $X\otimes (S)_{-1, -q}$, for $q> p+n-1,$ see \cite{HOUZ}.
 	
 	We rewrite all processes that figure in \eqref{NLJ} in their corresponding  Wiener-It\^o chaos expansion form and obtain 
 	\[\begin{split}\sum_{\alpha \in \mathcal I} \,\, \,\, \frac{d}{dt}u_\alpha(t) \,\, H_\alpha(\omega) &= \sum_{\alpha \in \mathcal I} A_\alpha u_\alpha(t) \, \, H_\alpha(\omega) +  \sum_{\alpha \in\mathcal{I}} \Big( \sum_{\gamma \leq  \alpha } \, u^{\lozenge n-1}_\gamma(t) \,\, u_{\alpha-\gamma} (t)\Big) \, H_\alpha(\omega) \\
 	&+\sum_{\alpha \in \mathcal I} f_\alpha(t) \, \, H_\alpha(\omega)  \\
 	\sum_{\alpha \in \mathcal I} u_\alpha(0) \ H_\alpha(\omega) & = \sum_{\alpha \in \mathcal I} u^0_\alpha \ H_\alpha(\omega).
 	\end{split}\] 
 Due to the orthogonality of the base $H_\alpha$ this reduces to the system of infinitely many deterministic Cauchy problems: 
 	\begin{enumerate}
 		\item[$1^\circ$] for $\alpha =\mathbf{0}$ 
 	\begin{equation}
 	\label{nelinearna det}
 	\frac{d}{dt} u_{\mathbf{0}} (t) =  A_{\mathbf{0}} u_{\mathbf{0}}  (t) +   u^n_{\mathbf{0}} (t) +f_\mathbf{0}(t), \quad u_{\mathbf{0}}(0) = u_{\mathbf{0}}^0,  \qquad \text{and}
 	\end{equation}
 	\item[$2^\circ$] for  $\alpha >\mathbf{0}$
 	\begin{equation}
 	\label{sistem 2}
 	\frac{d}{dt}u_\alpha (t)=  \big( A_\alpha +  n\,u^{n-1}_{\mathbf{0}} (t) \,Id \big) \, u_\alpha(t)  +  r_{\alpha,n} (t) + f_\alpha(t), \quad  
 	u_\alpha (0) =  u_\alpha^0 .
 	\end{equation} with $t\in (0,T]$ and $\omega\in\Omega$. 
 	\end{enumerate}
 	Let
 	\[B_{\alpha,n}(t) = A_\alpha +  n\,u^{n-1}_{\mathbf{0}}(t)\, Id \qquad  \text{and} \qquad g_{\alpha,n} (t)= r_{\alpha,n} (t) + f_\alpha(t), \quad t\in[0,T]\] for all $\alpha > \mathbf{0}$. Then, the system \eqref{sistem 2} 
 	can be written in the form 
 	\begin{equation}
 	\label{sistem 3}
 	\frac{d}{dt}u_\alpha (t) =  B_{\alpha,n}(t)  \, u_\alpha(t) +  g_{\alpha,n}(t), \quad t\in(0,T];\qquad  
 	u_\alpha (0) =  u_\alpha^0 .  
 	\end{equation}
 	 Note that the inhomogeneous part $g_{\alpha,n}$ in \eqref{sistem 3} does not contain any of the functions $u_\beta,\;\beta<\alpha$ for $|\alpha|=1$, while for $|\alpha|>1$ it involves also $u_\beta,\;\beta<\alpha$. Hence, we distinguish these two cases.
 \begin{enumerate}
 	\item[(a)] 
Let $|\alpha|=1$, i.e. $\alpha=\varepsilon_k$, $k\in \mathbb N.$ Then $g_{\varepsilon_k,n} =f_{\varepsilon_k}$, $k\in \mathbb N$ and thus \eqref{sistem 3} transforms to 
\begin{equation}
\label{det jed duzina 1}
\frac{d}{dt}u_{\varepsilon_k} (t)=  B_{\varepsilon_k,n}(t)  \, u_{\varepsilon_k} (t) + f_{\varepsilon_k}(t),\quad t\in(0,T];  \qquad  
u_{\varepsilon_k} (0) =  u_{\varepsilon_k}^0.
\end{equation}
\item[(b)] Let $|\alpha|>1.$ Then 
\begin{equation}
\nonumber
\frac{d}{dt}u_\alpha (t) =  B_{\alpha,n}(t)  \, u_\alpha(t) +  g_{\alpha,n}(t), \quad t\in(0,T];\qquad  
u_\alpha (0) =  u_\alpha^0.
\end{equation}
\end{enumerate}

Each solution $u$ to \eqref{NLJ} can be represented in the form \eqref{proces} and hence its coefficients $u_{\mathbf{0}} $ and $u_\alpha$  for  $\alpha> \mathbf{0}$ must satisfy   \eqref{nelinearna det} and  \eqref{sistem 3} respectively. Vice versa, if the coefficients $u_{\mathbf{0}} $ and $u_\alpha$  for  $\alpha> \mathbf{0}$ solve \eqref{nelinearna det} and  \eqref{sistem 3} respectively, and if the series in \eqref{proces} represented by these coefficients exists in $X\otimes (S)_{-1}$, then it defines a solution to \eqref{NLJ}.

\begin{definition} An $X-$valued generalized stochastic process $u(t)=\sum_{\alpha\in\mathcal{I}}u_\alpha(t)H_\alpha\in X\otimes (S)_{-1},\;t\in[0,T]$ is a coordinatewise classical solution to \eqref{NLJ} if $u_\mathbf{0}$ is a classical solution to \eqref{nelinearna det} and for every $\alpha\in\mathcal{I}\setminus\{\mathbf{0}\},$ the coefficient $u_\alpha$ is a classical solution to \eqref{sistem 3}. The coordinatewise solution $u(t)\in X\otimes (S)_{-1},\;t\in[0,T]$ is an almost classical solution to \eqref{NLJ} if $u\in C([0,T],X)\otimes (S)_{-1}.$ 
	An almost classical solution is a classical solution if $u\in C([0,T],X)\otimes (S)_{-1} \cap C^1((0,T],X)\otimes (S)_{-1}$.
\end{definition}

We assume that the following hold:
 	\begin{enumerate}
 		\item[(A1)] The operators  $A_\alpha,\;\alpha\in \mathcal{I},$ are infinitesimal generators of $C_0-$semigroups $\{T_\alpha (s)\}_{s\geq 0}$ with a common domain $D_{\alpha}=D,\;\alpha\in \mathcal{I},$ dense in $X.$ We assume that there exist constants $m\geq 1$ and $w\in \mathbb{R}$ such that
 		\begin{equation*} \|T_\alpha(s)\|\leq me^{w s},\;s\geq 0 \quad\mbox{for all}\quad\alpha\in \mathcal{I}.
 		\end{equation*}
 		 The action of  $\mathbf A$ is given by 
 		\[\mathbf A(u)=\sum_{\alpha\in\mathcal I}A_\alpha(u_\alpha)H_\alpha,\] for $u \in \mathbb{D}\subseteq D\otimes (S)_{-1}$ of the form \eqref{proces},  where
 		\[\mathbb{D}=\Big\{u=\sum_{\alpha\in\mathcal I}u_\alpha \, \,  H_\alpha\in
 		D\otimes (S)_{-1}:\;\exists p_0\geq 0,\; \sum_{\alpha\in\mathcal
 			I}\|A_\alpha(u_\alpha)\|^2_X(2\mathbb N)^{-p_0\alpha}<\infty\Big\}.\]
 			
 		 \item[(A2)] The initial value $u^0=\sum_{\alpha\in\mathcal{I}}u^0_\alpha H_\alpha\in\mathbb{D}$, i.e. $u_\alpha^0\in D$ for every $\alpha\in\mathcal{I}$ and there exists $p\geq 0$ such that 
 		\begin{equation*}\label{uslov1.1}
 		\sum_{\alpha\in \mathcal{I}}\|u_\alpha^0\|_X^2 (2\mathbb N)^{-p \alpha}<\infty, 
 		\end{equation*} 
 	\begin{equation*}\label{uslov1.2}
 	\sum_{\alpha\in \mathcal{I}}\|A_\alpha (u_\alpha^0)\|_X^2 (2\mathbb N)^{-p \alpha}<\infty.
 	\end{equation*}
 		
 		\item[(A3)] The inhomogeneous part $f(t,\omega)=\sum_{\alpha\in\mathcal{I}}f_\alpha(t)H_\alpha(\omega),\;t\in[0,T],\;\omega\in \Omega$ belongs to $ C^1([0,T],X)\otimes(S)_{-1};$ hence $t\mapsto f_\alpha(t)\in C^1([0,T],X),\;\alpha\in \mathcal{I}$ and there exists $p\geq 0$ such that
   $$\sum_{\alpha\in \mathcal{I}}\|f_\alpha\|^2_{C^1([0,T],X)}(2\mathbb{N})^{-p\alpha}=\sum_{\alpha\in \mathcal{I}}\Big(\sup_{t\in[0,T]}\|f_\alpha(t)\|_X+\sup_{t\in[0,T]}\|f'_\alpha(t)\|_X\Big)^2(2\mathbb{N})^{-p\alpha}<\infty.$$
 			
		\item[(A4-n)] The Cauchy problem
$$\frac{d}{dt} u_{\mathbf{0}} (t) =  A_{\mathbf{0}} u_{\mathbf{0}}  (t) +   u^{n}_{\mathbf{0}} (t) + f_\mathbf{0}(t) ,\quad t\in(0,T]; \quad u_{\mathbf{0}}(0) = u_{\mathbf{0}}^0,$$
has a classical solution $u_{\mathbf{0}}\in C^1([0,T],X).$ 
	\end{enumerate}

\begin{Remark}
Particularly, if $A_{\mathbf 0} = \Delta$ is  the Laplace operator and $f_{\mathbf{0}}\equiv 0$, then   \eqref{nelinearna det}  belongs to the class of Fujita equations  
\begin{equation}
\label{Fujita eq}
u_t= \Delta u + u^p, \quad u(0)=u_0,
\end{equation}studied by Fujita, Chen and Watanabe \cite{Fujita1, Fujita2}.  The authors proved that for a nonnegative initial condition $u^0\in C(\mathbb R^N)\cap L^\infty(\mathbb R^N)$,  equation \eqref{Fujita eq}  has a unique classical solution on some $[0,T_1)$. Moreover, if $p> 1+ \frac2{N}$ then there exist a positive bounded solution. The Fujita equation \eqref{Fujita eq} apart from an interest per se  also acts as a scaling limit of more general superlinear
equations whose nonlinearities exhibit a polynomial growth rate. Originally, it has been developed to describe molecular concentration of a solution subjected to centrifugation and sedimentation.
\end{Remark}

\begin{Remark}
In general, equations of the form \eqref{nelinearna det}, i.e. the deterministic equation for $\alpha = \mathbf{0}$ can be solved  by the Fixed Point Theorem \cite{FixP}.  Thus, in order to check if condition (A4-n) holds, one has to apply fixed point methods or other established methods for deterministic PDEs. The solution to \eqref{nelinearna det} will usually blow-up in finite time. Especially the description of blow-up in the Sobolev supercritical regime poses a challenge that has been tackled in several papers (e.g. \cite{Fujita2}, \cite{Meneses} for the Fujita equation). We stress that our equation \eqref{NLJ} and hence also \eqref{nelinearna det} is given on a finite time interval, which is assumed to provide a solution on the entire interval (we restrict our considerations form the very start to the interval where no blow-up appears).
\end{Remark}

Now we focus on solving \eqref{sistem 3} for $\alpha > \mathbf{0}$. 

\begin{Lemma}\label{Lema} Let the assumptions (A1)-(A4-n) be fulfilled. Then for every $\alpha>\mathbf{0}$ the evolution system \eqref{sistem 3} has a unique classical solution $u_\alpha\in C^1([0,T],X).$ 
\end{Lemma}

\begin{proof}
First, for every $\alpha> \mathbf{0},$ we consider the family of operators $B_{\alpha,n}(t) =  A_\alpha + n\, u^{n-1}_{\mathbf{0}}(t) Id,\;t\in[0,T].$  According to assumption (A1), the constant family $\{A_\alpha(t)\}_{t\in[0,T]}=\{A_\alpha\}_{t\in[0,T]}$ is a stable family of infinitesimal generators of a $C_0-$semigroup $\{T_\alpha(s)\}_{s\geq 0}$ on $X$ satisfying $\|T_\alpha(s)\|\leq me^{w s}$  with stability constants $m\geq 1$ and $w\in\mathbb{R}.$ Let 
\begin{equation}\label{ocena u^0}
M_n=\sup_{t\in[0,T]}\|u_0(t)\|_X.
\end{equation} 
The perturbation $n\,u^{n-1}_\mathbf{0}(t)Id:X\to X,\;t\in[0,T]$ is a family of uniformly bounded linear operators such that
$$\|n\,u^{n-1}_\mathbf{0}(t)x\|_X=\|n\,u^{n-1}_\mathbf{0}(t)\|_X\|x\|_X\leq \sup_{t\in[0,T]}n\,\|u_\mathbf{0}(t)\|^{n-1}_X\|x\|_X\leq nM_n^{n-1}\|x\|_X,$$
for all $x\in X,\; t\in[0,T],$ i.e. $\|n\,u^{n-1}_\mathbf{0}(t)Id\|\leq nM_n^{n-1},\;t\in [0,T].$ Thus, for every $\alpha>\mathbf{0},$ the family $\{A_\alpha+n\,u^{n-1}_{\mathbf{0}}(t)Id\}_{t\in[0,T]}$ is a stable family of infinitesimal generators with stability constants $m$ and $w+nM_n^{n-1}m.$ By assumption (A4-n) the function $u_\mathbf{0}\in C^1([0,T],X)$ so we obtain continuous differentiability of  $(A_\alpha+n\,u^{n-1}_\mathbf{0}(t)Id)x,\;t\in[0,T]$ for every $x\in D$ and for every $\alpha>\mathbf{0}.$ Additionally, the domain of the operators $n\,u^{n-1}_\mathbf{0}(t)Id$ is the entire space $X$ which implies that all of the operators $B_{\alpha,n}(t),\;t\in[0,T]$ have a common domain $D(B_{\alpha,n}(t))=D(A_\alpha)=D$ not depending on $t.$ Notice here that assumption (A1) additionally provides the same domain $D$ of the family $\{B_{\alpha,n}(t)\}_{t\in[0,T]}$ for all $\alpha>\mathbf{0}.$

Finally, one can associate the unique {\it evolution system} $ S_{\alpha,n}(t,s)$,  for $0 \leq s \leq t\leq T$ for all $\alpha> \mathbf{0}$ to the system \eqref{sistem 3} such that 
\begin{equation}
\label{ocena S_alpha}
\|S_{\alpha,n}(t,s)\| \leq  me^{w_n \, (t-s)} \leq me^{w_n (T-s)}, \quad  0 \leq s \leq t\leq T, \quad \alpha> \mathbf{0},
\end{equation}
where $w_n=w+nM_n^{n-1}m$ see \cite[Thm 4.8., p. 145]{Pazy}. Without loss of generality we may assume that $w>0$ and thus will be $w_n>0$.

Now one can solve the infinite system of the Cauchy problems \eqref{sistem 3} by induction on the length of the multiindex $\alpha$. Let $|\alpha|=1.$ Since $f_{\varepsilon_k}\in C^1([0,T],X),$ we obtain the unique classical  solution $u_{\varepsilon_k}\in C^1([0,T],X)$ to \eqref{det jed duzina 1} given by
\begin{equation}
\label{sol hom}
u_{\varepsilon_k} (t)= S_{\varepsilon_k,n} (t,0) \, u_{\varepsilon_k}^0 + \int_0^t \, S_{\varepsilon_k,n} (t,s) \, f_{\varepsilon_k}(s) \, ds, \quad t\in [0,T].
\end{equation}
Now let for every $\beta\in \mathcal{I}$ such that $\mathbf{0}<\beta<\alpha$ the unique classical solution of \eqref{sistem 3} satisfy $u_{\beta}\in C^1([0,T],X).$ Then for fixed $|\alpha|>1$ the inhomogeneous part $g_{\alpha,n}\in C^1([0,T],X)$
and the solution to \eqref{sistem 3}  is of the form
\begin{equation}
\label{sol nehom}
u_{\alpha}  (t) = S_{\alpha,n} (t,0) \, u_{\alpha}^0   + \int_0^t \, S_{\alpha,n} (t,s) \, g_{\alpha,n}(s) \, ds , \quad t\in [0,T],
\end{equation}
where $u_{\alpha}\in C^1([0,T],X).$ For more details see \cite[Thm 5.3., p. 147]{Pazy}. 
\end{proof}

Now we proceed with four technical lemmas that will be used in the sequel.

\begin{Lemma} \label{ocena faktorijela} Let $\alpha\in \mathcal{I}.$ Then 
 \begin{align*}
  \frac{|\alpha|!}{\alpha!}&\leq (2\mathbb{N})^{2\alpha}.
 \end{align*}
\end{Lemma}

\begin{proof} This is a direct consequence of \cite[Proposition 2.3]{KL}. More precisely, in \cite{KL} authors proved that $|\alpha|!\leq \mathbf{q}^\alpha \alpha!$ if a sequence $\mathbf{q}=(q_k)_{k\in \mathbb{N}}$ satisfies 
$$1<q_1\leq q_2\leq \dots\quad \mbox{and}\quad \sum_{k=1}^\infty\frac{1}{q_k}<1.$$ 
Since $\displaystyle\sum_{k=1}^\infty\frac{1}{(2k)^2}=\frac{\pi^2}{24}<1,$ the sequence $(2\mathbb{N})^2=((2k)^2)_{k\in\mathbb{N}}$ satisfies a required property. 
\end{proof}

\begin{Lemma} \label{lema ocena} For every $c>0$ there exists $q>1$ such that the following holds
\begin{align*}
\sum_{\alpha\in\mathcal{I}}c^{|\alpha|}(2\mathbb{N})^{-q\alpha}<\infty.
\end{align*}
\end{Lemma}

\begin{proof} 
Let $c>0$ and choose $s\geq 0$ such that $c\leq 2^s.$ Then, for $q>s+1,$
\begin{equation*}
\sum_{\alpha\in\mathcal{I}}c^{|\alpha|}(2\mathbb{N})^{-q\alpha}\leq \sum_{\alpha\in\mathcal{I}}\prod_{i=1}^\infty (2^s)^{\alpha_i}\prod_{i=1}^{\infty}(2 i)^{-q\alpha_i}\leq
\sum_{\alpha\in\mathcal{I}} \prod_{i=1}^\infty (2i)^{(s-q)\alpha_i}=\sum_{\alpha\in\mathcal{I}}(2\mathbb{N})^{(s-q)\alpha}<\infty.\qedhere
\end{equation*} 
\end{proof}

In the next lemma, for the sake of completeness, we give some useful properties of the well known Catalan numbers, see for example \cite{cat}.

\begin{Lemma}\label{katalan}
A sequence $\{\mathbf{c}_n\}_{n\in \mathbb N}$  defined by the recurrence relation
\begin{equation}
\label{Catalan}
\mathbf{c}_0=1, \quad 
\mathbf{c}_n=  \sum_{k=0}^{n-1} \, \mathbf{c}_k  \, \mathbf{c}_{n-1-k},  
\quad n\geq 1 
\end{equation}
is called  the sequence of Catalan numbers.  The closed formula for $\mathbf{c}_n$ is a multiple of the binomial coefficient, i.e.  the solution of the Catalan recurrence \eqref{Catalan} is 
\begin{equation*}\label{Catalan1}
\mathbf{c}_n= \frac1{n+1} \, \binom{2n}{n}\quad\mbox{or}\quad \mathbf{c}_n = \binom{2n}{n} - \binom{2n}{n+1}.
\end{equation*}
The Catalan numbers satisfy the growth estimate
\begin{equation}\label{Cat ocena}
\mathbf{c}_n\leq 4^n,\;n\geq 0.
\end{equation}
\end{Lemma}

\begin{Lemma} \label{multi katalan}\cite[p.21]{KL}
Let $\{R_\alpha:\;\alpha\in \mathcal{I}\}$ be a set of real numbers such that $R_\mathbf{0}=0,\;R_{\varepsilon_k},\; k\in\mathbb{N}$ are given and 
$$R_\alpha=\sum_{\mathbf{0}<\gamma<\alpha}R_\gamma R_{\alpha-\gamma},\quad |\alpha|>1.$$
Then
$$R_\alpha=\frac{1}{|\alpha|}\binom{2|\alpha|-2}{|\alpha|-1} \frac{|\alpha|!}{\alpha!}\prod_{k=1}^\infty R_{\varepsilon_k}^{\alpha_k},\quad |\alpha|>1.$$
\end{Lemma}

\begin{proof}
Let $\alpha\in\mathcal{I},\;|\alpha|>1$ be  given. Then $\alpha=(\alpha_1,\dots,\alpha_d,0,0,\dots)$ has only finally many non-zero components, so one can associate to it a $d-$dimensional vector $(\alpha_1,\dots,\alpha_d)\in \mathbb{N}_0^d.$ Adopting the proof for the classical Catalan numbers, the authors in \cite{KL} considered the function $G(z)=\sum_{\beta\in \mathbb{N}_0^d}M_\beta z^\beta,\;z\in \mathbb{N}_0^d,$ where $M_\beta=\sum_{\mathbf{0}<\gamma<\beta}M_\gamma M_{\beta-\gamma}$ and $z^\beta=z_1^{\beta_1}\cdots z_d^{\beta_d}.$ The function $G$ satisfies $G^2(z)-G(z)+\sum_{k=1}^d M_{\varepsilon_k}z_k=0,$ which implies that $G(z)=\sum_{n=1}^\infty \frac{1}{n}\binom{2n-2}{n-1}\left(\sum_{k=1}^dM_{\varepsilon_k}z_k\right)^n.$ Finally, applying the multinomial formula $\left(\sum_{k=1}^dM_{\varepsilon_k}z_k\right)^n=\sum_{\beta\in \mathbb{N}_0^d,\;|\beta|=n}\frac{n!}{\beta!}\prod_{k=1}^d\left(M_{\varepsilon_k}z_k\right)^{\beta_k}$ one obtains
\begin{align*}
G(z)&=\sum_{\beta\in \mathbb{N}_0^d}M_\beta z^\beta=\sum_{n=1}^\infty\sum_{\beta\in \mathbb{N}_0^d,\;|\beta|=n} \frac{1}{n}\binom{2n-2}{n-1}\frac{n!}{\beta!}\prod_{k=1}^d M_{\varepsilon_k}^{\beta_k}\prod_{k=1}^d z_k^{\beta_k}\\
&=\sum_{\beta\in \mathbb{N}_0^d} \left(\frac{1}{|\beta|}\binom{2|\beta|-2}{|\beta|-1}\frac{|\beta|!}{\beta!}\prod_{k=1}^d M_{\varepsilon_k}^{\beta_k}\right)z^\beta.
\end{align*}
\end{proof}

\subsection{Proof of the main theorem}

The statement of the main theorem is as follows.

\begin{Theorem}\label{main}
Let the assumptions $(A1)-(A4-n)$ be fulfilled. 
Then there exists a unique almost classical solution $u\in C([0,T],X)\otimes (S)_{-1}$ to \eqref{NLJ}. 
\end{Theorem}

\begin{proof} The proof of Theorem \ref{main} will be given by induction with respect to $n\in \mathbb{N}$ in Theorems \ref{T1} and \ref{T3}.
We will prove in the first one that the statement of the main theorem holds for $n=2.$ Since it is technically pretty challenging to write down the proof of the inductive step for arbitrary $n\in \mathbb{N},$ in Theorem \ref{T3} the proof is given for $n=3$ by reducing the problem to the case $n=2.$ In the same way one can reduce the problem for arbitrary $n\in \mathbb{N}$ to the case $n-1.$
\end{proof}
First consider \eqref{NLJ} for $n=2,$ i.e.
\begin{align}\label{NLJ-2}
u_t (t, \omega)&= \mathbf A \, u(t, \omega) + u^{\lozenge 2}(t, \omega)+f(t,\omega), \quad t\in[0,T]\\\nonumber
u(0,\omega) &= u^0(\omega), 
\end{align} 
The chaos expansion representation of the Wick-square is given by \eqref{WP2}. 
Applying the Wiener-It\^o chaos expansion to the nonlinear stochastic equation \eqref{NLJ-2} one obtain
 	\[\begin{split}\sum_{\alpha \in \mathcal I} \,\, \,\, \frac{d}{dt}u_\alpha(t) \,\, H_\alpha(\omega) &= \sum_{\alpha \in \mathcal I} A_\alpha u_\alpha(t) \, \, H_\alpha(\omega) +  \sum_{\alpha \in\mathcal{I}} \Big( \sum_{\gamma \leq  \alpha } \, u_\gamma(t) \,\, u_{\alpha-\gamma} (t)\Big) \, H_\alpha(\omega) \\
 	&+\sum_{\alpha \in \mathcal I} f_\alpha(t) \, \, H_\alpha(\omega)  \\
 	\sum_{\alpha \in \mathcal I} u_\alpha(0) \ H_\alpha(\omega) & = \sum_{\alpha \in \mathcal I} u^0_\alpha \ H_\alpha(\omega).
 	\end{split}\] 
which reduces to the system of infinitely many deterministic Cauchy problems:
 	\begin{enumerate}
 		\item[$1^\circ$] for $\alpha =\mathbf{0}$ 
 	\begin{equation}
 	\label{nelinearna det-2}
 	\frac{d}{dt} u_{\mathbf{0}} (t) =  A_{\mathbf{0}} u_{\mathbf{0}}  (t) +   u^2_{\mathbf{0}} (t) +f_\mathbf{0}(t), \quad u_{\mathbf{0}}(0) = u_{\mathbf{0}}^0,  \qquad \text{and}
 	\end{equation}
 	\item[$2^\circ$] for  $\alpha >\mathbf{0}$
 	\begin{equation}
 	\label{sistem 2-2}
 	\frac{d}{dt}u_\alpha (t)=  \big( A_\alpha +  2u_{\mathbf{0}} (t) \,Id \big) \, u_\alpha(t)  +  \sum_{\mathbf{0} < \gamma <  \alpha } \, u_\gamma(t) \,\, u_{\alpha-\gamma}(t) + f_\alpha(t), \quad  
 	u_\alpha (0) =  u_\alpha^0 .
 	\end{equation} with $t\in (0,T]$ and $\omega\in\Omega$. 
 	\end{enumerate}
Recall that
\begin{equation*}
\begin{split}
B_{\alpha,2}(t) = A_\alpha +  2u_{\mathbf{0}}(t)\, Id \qquad  \text{and} \qquad
g_{\alpha,2} (t) =  \sum_{\mathbf{0} < \gamma <  \alpha } \, u_\gamma(t) \,\, u_{\alpha-\gamma}(t) + f_\alpha(t), \quad t\in[0,T] 
\end{split}
\end{equation*}
for all $\alpha > \mathbf{0}$, so the system \eqref{sistem 2-2} 
can be written in the form 
\begin{equation}
\label{sistem 2a}
\frac{d}{dt}u_\alpha (t) =  B_{\alpha,2}(t)  \, u_\alpha(t) +  g_{\alpha,2}(t), \quad t\in(0,T];\qquad  
u_\alpha (0) =  u_\alpha^0 .  
\end{equation}

\begin{Theorem}\label{T1} Let the assumptions $(A1)-(A4-2)$ be fulfilled. 
Then there exists a unique almost classical solution $u\in C([0,T],X)\otimes (S)_{-1}$ to \eqref{NLJ-2}. 
\end{Theorem}

\begin{proof} According to Lemma \ref{Lema} for every $\alpha>\mathbf{0}$ the  evolution equation \eqref{sistem 2a} has an unique classical solution $u_\alpha\in C^1([0,T],X).$  Thus, the generalized stochastic process $u(t,\omega)=\sum_{\alpha\in\mathcal{I}}u_\alpha(t)H_\alpha(\omega),\;t\in[0,T],\;\omega\in \Omega$ has coefficients that are all classical solutions to the corresponding deterministic equation \eqref{sistem 2a}, hence in order to show that $u$ is an almost classical solution to \eqref{NLJ-2} one has to prove that $u\in C([0,T],X)\otimes (S)_{-1}.$ 

Let $u^0\in X\otimes (S)_{-1}$ be an initial condition satisfying assumption (A2) which states that there exist $\tilde{p}\geq 0$ and $\tilde{K}>0$ such that $\sum_{\alpha\in\mathcal{I}}\|u^0_\alpha\|_X^2(2\mathbb{N})^{-\tilde{p}\alpha}=\tilde{K}.$ Then there also exist $p\geq 0$ and $K\in (0,1)$ such that $\sum_{\alpha\in \mathcal{I}}\|u^0_\alpha\|_X^2(2\mathbb{N})^{-2p\alpha}=K^2$, or equivalently
\begin{equation}\label{ocena p. uslova}
(\exists p\geq 0)\;(\exists K\in (0,1))\; (\forall \alpha\in\mathcal{I})\quad \|u_\alpha^0\|_X\leq K(2\mathbb{N})^{p\alpha}.
\end{equation}
The inhomogeneous part $f\in C^1([0,T],X)\otimes (S)_{-1}$ satisfies assumption (A3) which states that there exists $\tilde{p}\geq 0$ such that $\sum_{\alpha\in\mathcal{I}}\sup_{t\in[0,T]}\|f_\alpha(t)\|_X^2(2\mathbb{N})^{-\tilde{p}\alpha}<\infty.$ Then there exist $p\geq 0$ and $K\in (0,1)$ such that
\begin{equation}\label{ocena ih. uslova}
\sup_{t\in[0,T]}\|f_\alpha(t)\|_X\leq K(2\mathbb{N})^{p\alpha},\quad \alpha\in \mathcal{I}.
\end{equation}
The coefficients $u_\alpha,\;\alpha\in\mathcal{I},\;\alpha>\mathbf{0}$ of the solution $u$ are given by \eqref{sol hom} and \eqref{sol nehom} for $n=2.$ Denote by
\begin{align*}
L_\alpha:=\sup_{t\in[0,T]}\|u_\alpha(t)\|_X,\quad \alpha\in\mathcal{I}.
\end{align*}

First, for $\alpha=\mathbf{0}$ using \eqref{ocena u^0} one obtain 
\begin{equation}\label{ocena L0-2}
L_\mathbf{0}=\sup_{t\in[0,T]}\|u_\mathbf{0}(t)\|_X=M_2,
\end{equation}
since the solution to \eqref{nelinearna det-2} satisfies assumption (A4-2).
Let $|\alpha|=1.$ Then $\alpha=\varepsilon_k,\;k\in \mathbb{N}$ and using \eqref{sol hom} we have that
$$\|u_{\varepsilon_k} (t)\|_X\leq \|S_{\varepsilon_k,2} (t,0)\|\| u_{\varepsilon_k}^0\|_X+\int_0^t\|S_{\varepsilon_k,2} (t,s)\|\|f_{\varepsilon_k}(s)\|_X ds, \quad t\in [0,T].$$
From \eqref{ocena S_alpha} we obtain that
\begin{align}\label{ocen int}
\int_0^t\|S_{\alpha,2}(t,s)\|ds\leq\int_0^t me^{w_2(t-s)}ds =m\frac{e^{w_2 t}-1}{w_2}\leq \frac{m}{w_2}e^{w_2 T},\quad t\in[0,T],\quad\alpha>\mathbf{0}
\end{align} and now \eqref{ocena S_alpha}, \eqref{ocena p. uslova} and \eqref{ocena ih. uslova} imply that
\begin{align}\label{ocena L_epsilon}
L_{\varepsilon_k}&=\sup_{t\in[0,T]}\|u_{\varepsilon_k} (t)\|_X\leq\sup_{t\in[0,T]}\Big\{\|S_{\varepsilon_k,2} (t,0)\|\| u_{\varepsilon_k}^0\|_X+\sup_{s\in[0,t]}\|f_{\varepsilon_k}(s)\|_X\int_0^t\|S_{\alpha,2}(t,s)\|ds\Big\}\\\nonumber
&\leq me^{w_2T}K(2\mathbb{N})^{p\varepsilon_k} +\frac{m}{w_2}e^{w_2 T}K(2\mathbb{N})^{p\varepsilon_k}=m_2e^{w_2T}K(2\mathbb{N})^{p\varepsilon_k},\quad t\in[0,T],\quad k\in\mathbb{N},
\end{align}
where $m_2=m+\frac{m}{w_2}.$

For $|\alpha|>1$ we consider two possibilities for $L_\alpha.$ First, if $L_\alpha\leq \sqrt{K}(2\mathbb{N})^{p\alpha}$ for all $|\alpha|>1$ then the statement of the theorem follows directly since for $q>2p+1$ and, having in mind \eqref{ocena L0-2} and \eqref{ocena L_epsilon}, we obtain
\begin{align*}
\sum_{\alpha\in\mathcal{I}}\sup_{t\in[0,T]}\|u_\alpha(t)\|^2_X(2\mathbb{N}&)^{-q\alpha}=\sum_{\alpha\in\mathcal{I}} L_\alpha^2(2\mathbb{N})^{-q\alpha}=L_\mathbf{0}^2+\sum_{k\in \mathbb{N}}L_{\varepsilon_k}^2(2\mathbb{N})^{-q\varepsilon_k} + \sum_{|\alpha|>1}L_\alpha^2(2\mathbb{N})^{-q\alpha}\\
&\leq M_2^2+(m_2e^{w_2 T}K)^2\sum_{k\in\mathbb{N}}(2\mathbb{N})^{(2p-q)\varepsilon_k} + K\sum_{|\alpha|>1}(2\mathbb{N})^{(2p-q)\alpha}<\infty,
\end{align*}
i.e. $u\in  C([0,T],X)\otimes (S)_{-1,-q}.$

In what follows, we will assume that $L_\alpha> \sqrt{K}(2\mathbb{N})^{p\alpha}$ for some $\alpha\in \mathcal{I},\;|\alpha|>1.$ Denote by $\mathcal I_*$ the set of all multi-indices  $\alpha\in \mathcal{I},\;|\alpha|>1$, for which $L_\alpha> \sqrt{K}(2\mathbb{N})^{p\alpha}$.
Then from \eqref{sol nehom} we obtain

$$u_\alpha(t)=S_{\alpha,2}(t,0)u^0_\alpha+\int_0^t S_{\alpha,2}(t,s)\Big[\sum_{0<\gamma<\alpha}u_{\alpha-\gamma}(s)u_\gamma(s)+f_\alpha(s)\Big]ds,\quad t\in [0,T].$$
From this we have
\begin{align*}
L_\alpha &=\sup_{t\in[0,T]}\|u_\alpha(t)\|_X \\
 & \leq \sup_{t\in[0,T]}\Bigg\{\|S_{\alpha,2}(t,0)\|\|u^0_\alpha\|_X+\int_0^t \|S_{\alpha,2}(t,s)\|\Big\|\sum_{0<\gamma<\alpha}u_{\alpha-\gamma}(s)u_\gamma(s)\Big\|ds\\
 &+\int_0^t\|S_{\alpha,2}(t,s)\|\|f_\alpha(s)\|_Xds\Bigg\}\\
&\leq \sup_{t\in[0,T]}\Bigg\{me^{w_2t}\|u^0_\alpha\|_X+\sup_{s\in[0,t]}\sum_{0<\gamma<\alpha}\|u_{\alpha-\gamma}(s)\|_X\|u_\gamma(s)\|_X\cdot\int_0^t \|S_{\alpha,2}(t,s)\|ds\\
&+\sup_{s\in[0,t]}\|f(s)\|_X\int_0^t\|S_{\alpha,2}(t,s)\|ds\Bigg\}.
\end{align*}
Using \eqref{ocen int} we obtain
\begin{align*}
L_\alpha &=\sup_{t\in[0,T]}\|u_\alpha(t)\|_X  \\
& \leq me^{w_2T}\|u^0_\alpha\|_X+\frac{m}{w_2}e^{w_2T}\sum_{0<\gamma<\alpha}\sup_{t\in[0,T]}\|u_{\alpha-\gamma}(t)\|_X\sup_{t\in[0,T]}\|u_\gamma(t)\|_X \\
&+\frac{m}{w_2}e^{w_2T}\sup_{s\in[0,t]}\|f(s)\|_X\\
& \leq m_2e^{w_2T}K(2\mathbb{N})^{p\alpha}+\frac{m}{w_2}e^{w_2T}\sum_{\mathbf{0}<\gamma<\alpha}L_{\alpha-\gamma}L_\gamma,
\end{align*} where again $m_2=m+\frac{m}{w_2}.$
Since $m_2\geq\frac{m}{w_2},$  one easily obtains
\begin{align}\label{ocena L_alpha}
L_\alpha\leq m_2e^{w_2T}\Big(K(2\mathbb{N})^{p\alpha}+\sum_{\mathbf{0}<\gamma<\alpha}L_{\alpha-\gamma}L_\gamma\Big).
\end{align}
Let $\tilde{L}_\alpha,\;\alpha>\mathbf{0},\;\alpha\in\mathcal I_*,$ be given by
\begin{align*}
\tilde{L}_{\alpha}:=2m_2e^{w_2 T}\Big(\frac{L_\alpha}{\sqrt{K}(2\mathbb{N})^{p\alpha}}+1\Big),\quad \alpha>\mathbf{0},\;\alpha\in\mathcal I_*.
\end{align*}
Thus, from \eqref{ocena L_epsilon} we have that for all $k\in \mathbb{N}$
\begin{align}\label{ocena tilda L_alpha}
\tilde{L}_{\varepsilon_k}=2m_2e^{w_2 T}\Big(\frac{L_{\varepsilon_k}}{\sqrt{K}(2\mathbb{N})^{p\varepsilon_k}}+1\Big)&\leq 2m_2e^{w_2 T}\Big(\frac{m_2e^{w_2 T}K(2\mathbb{N})^{p\varepsilon_k}}{\sqrt{K}(2\mathbb{N})^{p\varepsilon_k}}+1\Big)\\\nonumber
&=2m_2e^{w_2 T}(m_2e^{w_2 T}\sqrt{K}+1).
\end{align}
We proceed with the estimation of the term $\sum_{\mathbf{0}<\gamma<\alpha}\tilde L_\gamma\tilde L_{\alpha-\gamma}$ for given $|\alpha|>1,\;\alpha\in\mathcal I_*.$
\begin{align*}
\sum_{\mathbf{0}<\gamma<\alpha}\tilde L_\gamma\tilde L_{\alpha-\gamma}&=\sum_{\mathbf{0}<\gamma<\alpha} (2m_2e^{w_2 T})^2\Big(\frac{L_{\gamma}}{\sqrt{K}(2\mathbb{N})^{p\gamma}}+1\Big)\Big(\frac{L_{\alpha-\gamma}}{\sqrt{K}(2\mathbb{N})^{p(\alpha-\gamma)}}+1\Big)\\
&\geq (2m_2e^{w_2 T})^2 \Big(\sum_{\mathbf{0}<\gamma<\alpha}\frac{L_\gamma L_{\alpha-\gamma}}{K(2\mathbb{N})^{p\alpha}}+1\Big)\\
&=\frac{(2m_2e^{w_2 T})^2}{K(2\mathbb{N})^{p\alpha}}\sum_{\mathbf{0}<\gamma<\alpha}L_\gamma L_{\alpha-\gamma}+(2m_2e^{w_2 T})^2.
\end{align*}
Using inequality \eqref{ocena L_alpha} we obtain
\begin{align*}
\sum_{\mathbf{0}<\gamma<\alpha}\tilde L_\gamma\tilde L_{\alpha-\gamma}&\geq \frac{(2m_2e^{w_2 T})^2}{K(2\mathbb{N})^{p\alpha}}\Big(\frac{L_\alpha}{m_2e^{w_2 T}}-K(2\mathbb{N})^{p\alpha}\Big)+(2m_2e^{w_2 T})^2=\frac{4m_2e^{w_2 T}}{K(2\mathbb{N})^{p\alpha}}L_\alpha.
\end{align*}
Now since $L_\alpha>\sqrt{K}(2\mathbb{N})^{p\alpha}$ for $\alpha\in\mathcal I_*$ and since  $K<1$ we obtain
\begin{align*}
\sum_{\mathbf{0}<\gamma<\alpha}\tilde L_\gamma\tilde L_{\alpha-\gamma}&\geq \frac{4m_2e^{w_2 T}}{\sqrt{K}(2\mathbb{N})^{p\alpha}}L_\alpha=\frac{2m_2e^{w_2 T}}{\sqrt{K}(2\mathbb{N})^{p\alpha}}L_\alpha+\frac{2m_2e^{w_2 T}}{\sqrt{K}(2\mathbb{N})^{p\alpha}}L_\alpha\\
&\geq 2m_2e^{w_2 T}\Big(\frac{L_\alpha}{\sqrt{K}(2\mathbb{N})^{p\alpha}}+1\Big)=\tilde{L}_\alpha.
\end{align*}
Hence, for all $\alpha\in\mathcal I_*$, $|\alpha|>1$, we have obtained
\begin{align*}
\sum_{\mathbf{0}<\gamma<\alpha}\tilde L_\gamma\tilde L_{\alpha-\gamma}\geq \tilde{L}_\alpha. 
\end{align*}
Let $R_\alpha,\alpha>\mathbf{0},$ be defined as follows:
\begin{align*}
R_{\varepsilon_k}&=\tilde{L}_{\varepsilon_k},\quad k\in\mathbb{N},\\
R_\alpha&=\sum_{\mathbf{0}<\gamma<\alpha}R_\gamma R_{\alpha-\gamma},\quad |\alpha|>1.
\end{align*}
It is a direct consequence of the definition of the numbers $R_\alpha,\alpha>\mathbf{0},$ and it can be shown by induction with respect to the length of the multi-index $\alpha>\mathbf{0}$ that (see \cite[Section 5]{KL})
\begin{align}\label{L manje od R}
\tilde{L}_\alpha\leq R_{\alpha},\quad \alpha>\mathbf{0}.
\end{align}
Lemma \ref{multi katalan} shows that the numbers $R_\alpha,\;\alpha>\mathbf{0}$ satisfy
\begin{align*}
R_\alpha=\frac{1}{|\alpha|}\binom{2|\alpha|-2}{|\alpha|-1}\frac{|\alpha|!}{\alpha!}\prod_{i=1}^\infty R_{\varepsilon_i}^{\alpha_i},\quad \alpha>\mathbf{0}.
\end{align*}
Further on, by \eqref{ocena tilda L_alpha},
$$\prod_{i=1}^\infty R_{\varepsilon_i}^{\alpha_i}=\prod_{i=1}^\infty \tilde{L}_{\varepsilon_i}^{\alpha_i}\leq \prod_{i=1}^\infty (2m_2e^{w_2 T}(m_2e^{w_2 T}\sqrt{K}+1))^{\alpha_i}.$$
Let $c=2m_2e^{w_2 T}(m_2e^{w_2 T}\sqrt{K}+1).$ Then
\begin{align}\label{ocena R}
R_\alpha\leq \mathbf{c}_{|\alpha|-1}\frac{|\alpha|!}{\alpha!}c^{|\alpha|},\quad \alpha >\mathbf{0},
\end{align}
where $\mathbf{c}_n=\frac{1}{n+1}\binom{2n}{n},\;n\geq 0$ denotes the $n$th Catalan number (more information on Catalan numbers is provided in Lemma \ref{katalan}). Using Lemma \ref{ocena faktorijela}, \eqref{L manje od R}, \eqref{ocena R} and \eqref{Cat ocena} we obtain that for $\alpha\in\mathcal I_*$, $|\alpha|>1$ the estimation
\begin{align*}
\tilde{L}_\alpha\leq R_\alpha \leq 4^{|\alpha|-1}(2\mathbb{N})^{2\alpha}c^{|\alpha|}
\end{align*} holds.
Finally, from the definition of $\tilde L_\alpha,\;\alpha>\mathbf{0}$ we obtain
\begin{align*}
 L_\alpha\leq \Big(\frac{4^{|\alpha|-1}(2\mathbb{N})^{2\alpha}c^{|\alpha|}}{2m_2e^{w_2 T}}-1\Big)\sqrt{K}(2\mathbb{N})^{p\alpha}\leq \frac{\sqrt{K}}{8m_2e^{w_2 T}}(4c)^{|\alpha|}(2\mathbb{N})^{(p+2)\alpha}.
\end{align*}
Notice that the upper estimate also holds for $|\alpha|>1$, $\alpha\in\mathcal I\setminus\mathcal I_*$. Indeed, if $L_\alpha<\sqrt{K}(2\mathbb{N})^{p\alpha}$ then also $L_\alpha<\frac{\sqrt{K}}{8m_2e^{w_2 T}}(4c)^{|\alpha|}(2\mathbb{N})^{(p+2)\alpha}$, so we obtain
\begin{align*}
L_\alpha\leq \frac{\sqrt{K}}{8m_2e^{w_2 T}}(4c)^{|\alpha|}(2\mathbb{N})^{(p+2)\alpha},\quad \mbox{ for all } \alpha\in\mathcal I,\;\;|\alpha|>1 .
\end{align*}
Now we can prove that $u(t,\omega)=\sum_{\alpha\in \mathcal{I}}u_\alpha(t)H_\alpha(\omega)\in C([0,T],X)\otimes (S)_{-1}.$ Denote by $H=\frac{\sqrt{K}}{8m_2e^{w_2 T}}.$ Then
\begin{align*}
\sum_{\alpha\in\mathcal{I}}\sup_{t\in[0,T]}&\|u_\alpha(t)\|^2_X(2\mathbb{N})^{-q\alpha}=\sup_{t\in[0,T]}\|u_{\mathbf{0}}(t)\|^2_X+\sum_{\alpha>\mathbf{0}}\sup_{t\in[0,T]}\|u_\alpha(t)\|^2_X(2\mathbb{N})^{-q\alpha}\\
&=M_2^2+\sum_{k\in\mathbb{N}}L_{\varepsilon_k}^2(2\mathbb{N})^{-q\varepsilon_k}+\sum_{|\alpha|>1}L_\alpha^2(2\mathbb{N})^{-q\alpha}\\
&\leq M_2^2+(m_2e^{w_2 T}K)^2\sum_{k\in\mathbb{N}}(2\mathbb{N})^{(2p-q)\varepsilon_k}+H^2\sum_{|\alpha|>1}\Big((4c)^{|\alpha|}(2\mathbb{N})^{(p+2)\alpha}\Big)^2(2\mathbb{N})^{-q\alpha}\\
&=M_2^2+(m_2e^{w_2 T}K)^2\sum_{k\in\mathbb{N}}(2\mathbb{N})^{(2p-q)\varepsilon_k}+H^2\sum_{|\alpha|>1}(16c^2)^{|\alpha|}(2\mathbb{N})^{(2p+4-q)\alpha}.
\end{align*}
Taking that $s>0$ is such that $2^s\geq 16c^2,$ according to Lemma \ref{lema ocena}, we obtain
 \begin{align*}
\sum_{\alpha\in\mathcal{I}}\sup_{t\in[0,T]}\|u_\alpha(t)\|^2_X(2\mathbb{N})^{-q\alpha}&\leq M_2^2+(m_2e^{w_2 T}K)^2\sum_{k\in\mathbb{N}}(2\mathbb{N})^{(2p-q)\varepsilon_k}\\
&+H^2\sum_{|\alpha|>1}(2\mathbb{N})^{(2p+4+s-q)\alpha}<\infty
\end{align*}
for $q>2p+s+5.$
\end{proof}

In the sequel we prove the existence of the almost classical solution of the Cauchy problem 
\begin{align}\label{NLJ-3}
u_t (t, \omega)&= \mathbf A \, u(t, \omega) + u^{\lozenge 3}(t, \omega)+f(t,\omega), \quad t\in[0,T]\\\nonumber
u(0,\omega) &= u^0(\omega), 
\end{align} 

Note that
 \begin{align}
 &u^{\lozenge 3}(t,\omega)= u^{\lozenge 2}(t,\omega) \, \lozenge \, u(t,\omega) 
 =\sum_{\alpha \in \mathcal I}  \, \sum_{\beta\leq \alpha }  \sum_{\gamma\leq \beta } \,  u_{\alpha-\beta} (t) \,  u_{\beta-\gamma} (t) \, u_\gamma(t) \,  H_\alpha(\omega) \,\nonumber \\
 &=u^3_{\mathbf{0}}(t)\,H_{\mathbf{0}}(\omega) \,\nonumber\\
 &+  \sum_{|\alpha|>0} \Big(3 u_{\mathbf{0}}^2 u_\alpha(t)+ 3 u_{\mathbf{0}} \sum_{0<\beta< \alpha }  u_{\alpha-\beta} (t)  u_\beta(t) + 
\sum_{0<\beta < \alpha } \sum_{0<\gamma< \beta }    u_{\alpha-\beta} (t) u_{\beta-\gamma} (t)   u_\gamma(t)\Big)  H_\alpha(\omega), 
\label{WP3}
 \end{align} for $ t\in[0,T]$, $\omega\in\Omega$.  
Applying the Wiener-It\^o chaos expansion method to the nonlinear stochastic equation \eqref{NLJ-3} reduces to the system of infinitely many deterministic Cauchy problems: 
\begin{enumerate}
	\item[$1^\circ$] for $\alpha =\mathbf{0}$ 
	\begin{equation*}
	\frac{d}{dt} u_{\mathbf{0}} (t) =  A_{\mathbf{0}} u_{\mathbf{0}}  (t) +   u^3_{\mathbf{0}} (t) +f_\mathbf{0}(t), \quad u_{\mathbf{0}}(0) = u_{\mathbf{0}}^0,  \qquad \text{and}
	\end{equation*}
	\item[$2^\circ$] for  $\alpha >\mathbf{0}$
	\begin{equation}
	\label{sistem alfa -jed treceg reda}
	\begin{split}
	\frac{d}{dt}u_\alpha (t) &=  \big( A_\alpha +  3u_{\mathbf{0}}^2 (t) Id \big) u_\alpha(t)  +  3 u_{\mathbf{0}} \sum_{0<\beta< \alpha }  u_{\alpha-\beta} (t)  u_\beta(t) + \\	
	  &+\sum_{0<\beta < \alpha } \sum_{0<\gamma< \beta }    u_{\alpha-\beta} (t) u_{\beta-\gamma} (t)    u_\gamma(t) + f_\alpha(t), \\
	u_\alpha (0) &=  u_\alpha^0 .
\end{split}
	\end{equation} with $t\in (0,T]$ and $\omega\in\Omega$. 
\end{enumerate}
Let
\begin{equation}
\begin{split}
B_{\alpha,3}(t) &= A_\alpha +  3u_{\mathbf{0}}^2(t)\, Id \qquad  \text{and}  \\
g_{\alpha,3} (t) &= 3 u_{\mathbf{0}} \sum_{0<\beta< \alpha }  u_{\alpha-\beta} (t)  u_\beta(t) + 
\sum_{0<\beta < \alpha } \sum_{0<\gamma< \beta }    u_{\alpha-\beta} (t) u_{\beta-\gamma} (t) u_\gamma(t) + f_\alpha(t), \quad t\in[0,T] \label{h alfa}
\end{split}
\end{equation}
for all $\alpha > \mathbf{0}$, then, the system \eqref{sistem alfa -jed treceg reda} 
can be written in the form 
\begin{equation}
\label{sistem 3a}
\frac{d}{dt}u_\alpha (t) =  B_{\alpha,3}(t)  \, u_\alpha(t) +  g_{\alpha,3}(t), \quad t\in(0,T];\qquad  
u_\alpha (0) =  u_\alpha^0 .  
\end{equation}

\begin{Theorem}\label{T3} Let the assumptions $(A1)-(A4-3)$ be fulfilled. 
	Then, there exists a unique almost classical solution $u\in C([0,T],X)\otimes (S)_{-1}$ to \eqref{NLJ-3}.
\end{Theorem}

\begin{proof} 
	According to Lemma \ref{Lema} for every $\alpha>\mathbf{0}$ the  evolution equation \eqref{sistem 3a} has an unique classical solution $u_\alpha\in C^1([0,T],X)$ given in the form \eqref{sol nehom}.   Thus, the generalized stochastic process $u(t, \omega)$, represented in the chaos expansion form \eqref{proces},  
	has coefficients that are all classical solutions to the corresponding deterministic equation \eqref{sistem 3a}. Hence,  in order to show that $u$ is an almost classical solution to \eqref{NLJ-3}, one has to prove that $u\in C([0,T],X)\otimes (S)_{-1}.$ 
	
	We assume that the initial condition $u^0\in X\otimes (S)_{-1}$ satisfies assumption $(A2)$, i.e. the estimate  \eqref{ocena p. uslova} holds true. 
	The inhomogeneous part $f\in C^1([0,T],X)\otimes (S)_{-1}$ satisfies assumption $(A3)$, i.e. 
	the estimate \eqref{ocena ih. uslova} is true for some $p\geq 0$. 
	Moreover, the coefficients $u_\alpha,\;\alpha\in\mathcal{I},\;\alpha>\mathbf{0}$ of the solution $u$ are given by \eqref{sol hom} and \eqref{sol nehom} for $n=3.$ Now, for all $\alpha\in\mathcal{I}$ we are going to  estimate
	\begin{align*}
	L_\alpha =\sup_{t\in[0,T]}\|u_\alpha(t)\|_X .
	\end{align*}
	
	It is clear that for $\alpha=\mathbf{0}$ , by $(A4-3)$   we have
	$L_{\mathbf{0}} =  \sup_{t\in[0,T]} \|u_{\mathbf{0}}(t)\| = M_3$. 
	
	For, $|\alpha|=1$, i.e. for $\alpha=\varepsilon_k$, $k\in \mathbb{N}$  by \eqref{sol hom} we have that
	$$\|u_{\varepsilon_k} (t)\|_X\leq \|S_{\varepsilon_k,3} (t,0)\|\| u_{\varepsilon_k}^0\|_X+\int_0^t\|S_{\varepsilon_k,3} (t,s)\|\|f_{\varepsilon_k}(s)\|_X ds, \quad t\in [0,T].$$
	From \eqref{ocena S_alpha} we obtain that
	\begin{align}\label{ocen int1}
	\int_0^t\|S_{\alpha,3}(t,s)\|ds\leq\int_0^t me^{w_3(t-s)}ds 
	\leq \frac{m}{w_3}e^{w_3 T},\quad t\in[0,T],\quad\alpha>\mathbf{0}.
	\end{align} By \eqref{ocena p. uslova},  \eqref{ocena ih. uslova}, \eqref{ocena S_alpha} and \eqref{ocen int1} we obtain 
	\begin{align*}
	L_{\varepsilon_k}&=\sup_{t\in[0,T]}\|u_{\varepsilon_k} (t)\|_X\leq\sup_{t\in[0,T]}\Big\{\|S_{\varepsilon_k,3} (t,0)\|\| u_{\varepsilon_k}^0\|_X+\sup_{s\in[0,t]}\|f_{\varepsilon_k}(s)\|_X\int_0^t\|S_{\alpha,3}(t,s)\|ds\Big\}\\\nonumber
	&\leq me^{w_3T}K(2\mathbb{N})^{p\varepsilon_k} +\frac{m}{w_3}e^{w_3 T}K(2\mathbb{N})^{p\varepsilon_k}, 
	\end{align*}which leads to the estimate 
	\begin{equation}
	\label{ocena L_epsilon kub}
L_{\varepsilon_k} \leq 	m_3e^{w_3T}K(2\mathbb{N})^{p\varepsilon_k},\quad  k\in\mathbb{N},
	\end{equation}
	where $m_3=m+\frac{m}{w_3}.$
	
	For $|\alpha|=2$ we have two different forms of the multiindex. First, for $\alpha=2\varepsilon_k$, $k\in \mathbb N$ from \eqref{h alfa} we obtain the form of the inhomogeneous part  $g_{2\varepsilon_k,3}(t)= 3 u_{\mathbf{0}}(t) \, u_{\varepsilon_k}^2(t)+f_{2\varepsilon_k}(t)$, where 
	\[\begin{split}
	\sup_{s\in[0,t]}\|g_{2\varepsilon_k,3}(s)\|_X &\leq 3M_3 L_{\varepsilon_k}^2 + \sup_{s\in[0,t]}\|f_{2\varepsilon_k}(s)\|_X\\
	& \leq  3M_3 m_3^2e^{2w_3T}K^2 (2\mathbb{N})^{2p\varepsilon_k} + K (2\mathbb{N})^{2p\varepsilon_k} \\
	&\leq  (3M_3 m_3^2e^{2w_3T}K^2 + K) \, (2\mathbb{N})^{2p\varepsilon_k}. 
	\end{split}\]
	
	 Then, together with  \eqref{sol nehom} we obtain 
		\begin{align*}
		L_{2\varepsilon_k}&=\sup_{t\in[0,T]}\|u_{2\varepsilon_k} (t)\|_X \\
		&\leq\sup_{t\in[0,T]}\Big\{\|S_{2\varepsilon_k,3} (t,0)\|\| u_{2\varepsilon_k}^0\|_X+\sup_{s\in[0,t]}\|g_{2\varepsilon_k,3}(s)\|_X\int_0^t\|S_{2\varepsilon_k,3}(t,s)\|ds\Big\}\\\nonumber
		&\leq me^{w_3T}K(2\mathbb{N})^{2p\varepsilon_k} +\frac{m}{w_3}e^{w_3 T} \, (3M_3 m_3^2e^{2w_3T}K^2 + K) \, (2\mathbb{N})^{2p\varepsilon_k}. 
		\end{align*}
		Thus, 
		\begin{equation}
		\label{L 2 epsilon k}
		L_{2\varepsilon_k} \leq a_1 \, e^{w_3T} K \, (2\mathbb{N})^{2p\varepsilon_k}, \quad  k\in\mathbb{N},
		\end{equation}
		where $a_1 = m + \frac{m}{w_3} (3M_3 m_3^2 e^{2w_3T} K +1)$.

In the second case, for  $\alpha= \varepsilon_k + \varepsilon_j$, $k\not = j$, $k,j\in \mathbb N$ from \eqref{h alfa} we obtain the form   $g_{\varepsilon_k + \varepsilon_j,3}(t) = 6 u_{\mathbf{0}}(t) \, u_{\varepsilon_k }(t) \, u_{\varepsilon_j }(t) + f_{\varepsilon_k + \varepsilon_j}(t)$ of the inhomogeneous part of \eqref{sol nehom}. By applying \eqref{ocena L_epsilon kub} and \eqref{ocena ih. uslova} 
it can be estimated as
	\[\begin{split}
	\sup_{s\in[0,t]}\|g_{\varepsilon_k+ \varepsilon_j,3}(s)\|_X &\leq 6M_3 L_{\varepsilon_k}  L_{\varepsilon_j} + \sup_{s\in[0,t]}\|f_{\varepsilon_k+ \varepsilon_j}(s)\|_X\\
	& \leq  6M_3 m_3^2e^{2w_3T}K^2(2\mathbb{N})^{p\varepsilon_k} (2\mathbb{N})^{p\varepsilon_j} + K (2\mathbb{N})^{p\varepsilon_k+ p \varepsilon_j}\\
	& \leq  (6M_3 m_3^2e^{2w_3T}K^2 + K) \, (2\mathbb{N})^{p(\varepsilon_k+ \varepsilon_j)}.
	\end{split}\]  
Then,  \eqref{sol nehom} combined with the previous estimate 	lead to 
	\begin{align*}
	L_{\varepsilon_k+ \varepsilon_j}&=\sup_{t\in[0,T]}\|u_{\varepsilon_k+ \varepsilon_j} (t)\|_X\\
	&\leq\sup_{t\in[0,T]}\Big\{\|S_{\varepsilon_k+ \varepsilon_j,3} (t,0)\|\| u_{\varepsilon_k+ \varepsilon_j}^0\|_X+\sup_{s\in[0,t]}\|g_{\varepsilon_k+ \varepsilon_j,3}(s)\|_X\int_0^t\|S_{\varepsilon_k+ \varepsilon_j,3}(t,s)\|ds\Big\}\\\nonumber
	&\leq me^{w_3T}K(2\mathbb{N})^{p(\varepsilon_k+ \varepsilon_j)} +\frac{m}{w_3}e^{w_3 T} \, (6M_3 m_3^2e^{2w_3T}K^2 + K) \, (2\mathbb{N})^{p(\varepsilon_k+ \varepsilon_j)}. 
	\end{align*}
	Then, we obtained
\begin{equation}
\label{ocena L eps_k + eps_j}
L_{\varepsilon_k+ \varepsilon_j} \leq a_2 \, e^{w_3T} K \, (2\mathbb{N})^{p(\varepsilon_k+ \varepsilon_j)}, \quad k,j\in\mathbb{N},\; k\not=j,
\end{equation}
where $a_2 = m + \frac{m}{w_3} (6M_3 m_3^2 e^{2w_3T} K +1)$.
Finaly, from \eqref{L 2 epsilon k} and \eqref{ocena L eps_k + eps_j} we obtain the estimate for all  $|\alpha|=2$
\begin{equation*}
\label{ocena Lalfa duzine 2}
L_\alpha \leq a_2 \, e^{w_3T} K \, (2\mathbb{N})^{p\alpha}.  
\end{equation*}
	
	For $|\alpha|>2$ we deal with  general  form of the inhomogeneous part of \eqref{sistem 3a} \[g_{\alpha,3}(t) = 3 u_{\mathbf{0}} \sum_{0<\beta< \alpha }  u_{\alpha-\beta} (t)  u_\beta(t) + 
	\sum_{0<\beta < \alpha } \sum_{0<\gamma< \beta }    u_{\alpha-\beta} (t) \, u_{\beta-\gamma} (t)  \,  u_\gamma(t) + f_\alpha(t), \quad t\in [0,T].\] 
	The solution to \eqref{sistem 3a} is of the form
	\begin{align*}&u_\alpha(t)= S_{\alpha,3}(t,0) u_\alpha^0 \\
	&+ \int_0^t S_{\alpha,3}(t,s) \Big( 3 u_{\mathbf{0}} \sum_{0<\beta< \alpha }  u_{\alpha-\beta} (t)  u_\beta(t) +
	&\sum_{0<\beta < \alpha } \sum_{0<\gamma< \beta }   u_{\alpha-\beta} (t)  u_{\beta-\gamma} (t)   u_\gamma(t) + f_\alpha(t)\Big) ds. \end{align*}
	We underline that in the  previous inductive steps, we obtained the estimates of $L_{\alpha-\theta}= \sup_{t\in[0,T]}\|u_{\alpha-\theta}(t)\|$ for all $\mathbf{0}<\theta < \alpha$.  Then, 
	\begin{equation}
	\begin{split}
	L_\alpha = \sup_{t\in[0,T]}&\|u_{\alpha}(t)\| \leq me^{\omega_3T} K(2\mathbb N)^{p\alpha} \\
	&+ \frac{m}{w_3} \Big(3M_3 \sum_{0<\beta< \alpha }  L_{\alpha-\beta}  L_\beta + 
	\sum_{0<\beta < \alpha } \sum_{0<\gamma< \beta }    L_{\alpha-\beta} L_{\beta-\gamma}    L_\gamma + K (2\mathbb N)^{p\alpha}\Big)\\
	&\leq m_3 e^{\omega_3T}  \Big(K(2\mathbb N)^{p\alpha} + 3M_3 \sum_{0<\beta< \alpha }  L_{\alpha-\beta}  L_\beta + 
	\sum_{0<\beta < \alpha } L_{\alpha-\beta} \, \sum_{0<\gamma< \beta }     L_{\beta-\gamma}  \,  L_\gamma \Big), \label{ocena Lalfa} 
	\end{split}
	\end{equation} where $m_3=m + \frac{m}{w_3}$.   \\
	In order to estimate $L_\alpha$ for $|\alpha|>2$ we consider two possibilities: (a) $L_\alpha \leq \sum\limits_{0<\beta < \alpha } L_{\alpha-\beta} \, L_\beta,$ $|\alpha|>2$ and (b) $L_\alpha > \sum\limits_{0<\beta < \alpha } L_{\alpha-\beta} \, L_\beta,\;|\alpha|>2.$
	
	\begin{enumerate}
	\item[(a)] Define $R_\alpha$ for $|\alpha|\geq 1$ in the following inductive way 
	\begin{equation*}
	\begin{split}
	R_{\varepsilon_k} &= L_{\varepsilon_k}\\
	R_\alpha &= \sum\limits_{0<\beta < \alpha } R_{\alpha-\beta} \, R_\beta, \quad |\alpha|\geq 2, 
	\end{split}
	\end{equation*}then, using Lemma \ref{multi katalan}, we obtain  the estimate 
	\begin{equation*}
	L_\alpha \leq R_\alpha = \frac1{|\alpha|} \,  \binom{2|\alpha|-2}{|\alpha|-1} \,\,  \frac{|\alpha|!}{\alpha!} \, 
\big(\prod_{i=1}^{\infty} \, R_{\varepsilon_i}^{\alpha_i} \big). 
	\end{equation*}
Moreover, by \eqref{ocena L_epsilon kub} we get 
	\begin{equation*}
	\begin{split}
	\prod_{i=1}^{\infty} \, R_{\varepsilon_i}^{\alpha_i} &= \prod_{i=1}^{\infty} \, L_{\varepsilon_i}^{\alpha_i}
\leq \prod_{i=1}^{\infty} \Big(  m_3 e^{\omega_3T} K\, (2\mathbb N)^{p \varepsilon_k}\Big)^{\alpha_i}  = 
	\big(  m_3 e^{\omega_3T} K \big)^{|\alpha|} \, \,  \prod_{i=1}^{\infty} \,  (2i)^{p \alpha_i} \\
	& = \big(  m_3 e^{\omega_3T} K \big)^{|\alpha|} \, \,  (2\mathbb N)^{p\alpha} = c_3^{|\alpha|} (2\mathbb N)^{p\alpha},
\end{split}
\end{equation*} where $c_3 = m_3 e^{\omega_3T} K$. We also used $\prod_{i=1}^{\infty} \,  (2i)^{p \alpha_i} =  (2\mathbb N)^{p\alpha} $
and $(2\mathbb N)^{\varepsilon_i} = 2i  $.
We recall the form of the   Catalan numbers $\mathbf{c}_{|\alpha|} = \frac1{|\alpha|}	\binom{2|\alpha|-2}{|\alpha|-1} $, $|\alpha|\geq 2$. Then, by Lemma \ref{ocena faktorijela} we obtain
\begin{align*}
L_\alpha  &\leq \frac1{|\alpha|}	\binom{2|\alpha|-2}{|\alpha|-1}  \, \frac{|\alpha|!}{\alpha!} \, \,  c_3^{|\alpha|} \, (2\mathbb N)^{p\alpha}
\leq 4^{|\alpha|-1} \, (2\mathbb N)^{2\alpha} \, c_3^{|\alpha|} \, (2\mathbb N)^{p\alpha}\\
& \leq (2\mathbb N)^{p_3\alpha} (2\mathbb N)^{(2+p)\alpha}, \nonumber
\end{align*} where we used that  $4^{|\alpha|-1} \, 	 c_3^{|\alpha|}\leq (2\mathbb N)^{p_3\alpha}$ for some positive $p_3$.  Thus, we conclude 
\begin{equation*}
\label{Lalfa case 1}
L_\alpha \leq (2\mathbb N)^{(p_3+p+2)\alpha}. 
 \end{equation*}
Finally, for $q>2p_3+ 2p +5$ the statement of the theorem follows  from  
	\begin{align}
	\label{b}
\sum_{\alpha\in\mathcal{I}}\sup_{t\in[0,T]}\|u_\alpha(t)\|^2_X(2\mathbb{N})^{-q\alpha} & = \sum_{\alpha\in\mathcal{I}} L_\alpha^2 (2\mathbb{N})^{-q\alpha} \nonumber\\
&=L_\mathbf{0}^2+\sum_{k\in \mathbb{N}}L_{\varepsilon_k}^2(2\mathbb{N})^{-q\varepsilon_k} + \sum_{|\alpha|>1}L_\alpha^2(2\mathbb{N})^{-q\alpha} \nonumber\\
	&\leq M_3^2+(m_3e^{w_3 T}K)^2 \sum_{k\in\mathbb{N}}(2\mathbb{N})^{(2p-q)\varepsilon_k} \nonumber\\
	&+ \sum_{|\alpha|>1}(2\mathbb{N})^{(2(p_3+p+2)-q)\alpha}<\infty, 
	\end{align}
 i.e. $u\in  C([0,T],X)\otimes (S)_{-1,-q}$. Note  that in \eqref{b} the term $\sum_{k\in\mathbb{N}}(2\mathbb{N})^{(2p-q)\varepsilon_k}$ is finite since   $q>2p+1$ when $q>2p_3+ 2p +5$.  
\item[(b)]  We assume, in the second case, that there exists $\alpha\in \mathcal I$, $|\alpha| \geq 2$ such that 
\begin{equation}
\label{a} 
L_\alpha > \sum\limits_{0<\beta < \alpha } L_{\alpha-\beta} \, L_\beta. 
\end{equation} 
Consider the most complicated case.  Then, we would have that  the inequality \eqref{a} is fulfilled for all $\alpha\in \mathcal I$.  Then, \eqref{ocena Lalfa} reduces to
\begin{equation*}
L_\alpha \leq m_3 e^{w_3T} \left(K(2\mathbb N)^{p\alpha} + (3M_3+1) \, \sum_{0< \beta< \alpha }     L_{\alpha-\beta}  \,  L_\beta\right),
\end{equation*}where we used inequality $L_{\beta}>\sum\limits_{0<\gamma<\beta } L_{\beta-\gamma} \, L_\gamma$ for $\beta < \alpha$.   Further, we have
\begin{equation*}
L_\alpha \leq  (3M_3+1) \,  m_3 e^{w_3T} \, \Big( \frac{K}{3M_3+1} \, (2\mathbb N)^{p\alpha} \, + \sum\limits_{0<\beta < \alpha } L_{\alpha-\beta} \, L_\beta  \Big), \quad |\alpha|\geq 2. 
\end{equation*}At this point, we can repeat the proof of Theorem \ref{T1}.  Particularly, using the notation $m_3'=(3M_3+1) \,  m_3$ and $K'=\frac{K}{3M_3+1},$ the following inequality 
\begin{align*} 
L_\alpha\leq m_3'e^{w_3T}\Big(K'(2\mathbb{N})^{p\alpha}+\sum_{\mathbf{0}<\beta<\alpha}L_{\alpha-\beta}L_\beta\Big) 
\end{align*}
corresponds to the inequality \eqref{ocena L_alpha}, since $K'<1,$ and the proof  continues in the same manner  as the one from Theorem \ref{T1}, i.e. the proof of solvability of the  equation \eqref{NLJ-2} with the Wick-square nonlinearity. 
	\end{enumerate}
	\end{proof} 

\begin{Remark}
Note here that if the almost classical solution $u$ to \eqref{NLJ} satisfies $u\in \mathbb{D}=Dom{\mathbf{A}}$ then $u$ is a classical solution to \eqref{NLJ}.
\end{Remark}

\subsection{The linear nonautonomous case}

Our analysis provides a downright observation for the linear nonautonomous equation
\begin{align}\label{LJ}
u_t (t, \omega)&= \mathbf A(t) \, u(t, \omega) + f(t,\omega), \quad t\in(0,T]\\\nonumber
u(0,\omega) &= u^0(\omega),\quad \omega\in\Omega. 
\end{align} 
We assume the following:
	\begin{enumerate}
		
		\item[(B1)] The operator $\mathbf A(t):\;\mathbb{D}'\subset X\otimes(S)_{-1}\to X\otimes(S)_{-1},\;t\in[0,T]$ is a coordinatewise operator depending on $t$ that corresponds to a family of deterministic operators $A_\alpha(t):D(A_\alpha)\subset X\to X,\;\alpha\in\mathcal{I}.$ For every $\alpha\in\mathcal{I}$ the operator family $\{A_\alpha(t)\}_{t\in[0,T]}$ is a stable family of infinitesimal generators of $C_0-$semigroups on $X$ with stability constants $m>1$ and $w\in\mathbb{R}$ not depending on $\alpha,$ therefore the  corresponding evolution systems $S_\alpha(t,s)$ satisfy
$$\|S_\alpha(t,s)\|\leq me^{w(t-s)}\leq me^{wT},\quad0\leq s<t\leq T,\quad \alpha\in\mathcal{I}.$$
   The domain $D(A_\alpha(t))=D$ is independent of $t\in[0,T]$ and $\alpha\in\mathcal{I}.$ For every $x\in D$ the function $A_\alpha(t)x,\;t\in[0,T]$ is continuously differentiable in $X$ for each $\alpha\in\mathcal{I}.$  
   
   The action of  $\mathbf A(t),\;t\in [0,T]$ is given by 
 		\[\mathbf A(t)(u)=\sum_{\alpha\in\mathcal I}A_\alpha(t)(u_\alpha)H_\alpha,\] for $u \in \mathbb{D}'\subseteq D\otimes (S)_{-1}$ of the form \eqref{proces},  where
 		\[\mathbb{D}'=\Big\{u=\sum_{\alpha\in\mathcal I}u_\alpha H_\alpha\in
 		D\otimes (S)_{-1}:\exists p_0\geq 0, \sum_{\alpha\in\mathcal
 			I}\sup_{t\in[0,T]}\|A_\alpha(t)(u_\alpha)\|^2_X(2\mathbb N)^{-p_0\alpha}<\infty\Big\}.\]
 			
 		\item[(B2)] The initial value $u^0=\sum_{\alpha\in\mathcal{I}}u^0_\alpha H_\alpha\in\mathbb{D}'$, i.e. $u_\alpha^0\in D$ for every $\alpha\in\mathcal{I}$ and there exists $p\geq 0$ such that 
 		\begin{equation*}
 		\sum_{\alpha\in \mathcal{I}}\|u_\alpha^0\|_X^2 (2\mathbb N)^{-p \alpha}<\infty, 
 		\end{equation*} 
 	\begin{equation*}
 	\sum_{\alpha\in \mathcal{I}}\sup_{t\in[0,T]}\|A_\alpha(t) u_\alpha^0\|_X^2 (2\mathbb N)^{-p \alpha}<\infty.
 	\end{equation*}
\end{enumerate} 
		
		For the inhomogeneous part $f(t,\omega),\;\omega\in \Omega,\;t\in[0,T]$ we assume (A3).

\begin{Theorem}\label{T2}
Let the assumptions $(B1),\;(B2)$ and $(A3)$ be fulfilled. 
Then there exists a unique almost classical solution $u\in C([0,T],X)\otimes (S)_{-1}$ to \eqref{LJ}.
\end{Theorem}

\begin{proof}
Applying the Wiener-It\^o chaos expansion method to \eqref{LJ} we obtain the system of infinitely many deterministic Cauchy problems
\begin{align}\label{CP za LJ}
\frac{d}{dt}u_\alpha(t)&=A_\alpha(t)u_\alpha(t)+f_\alpha(t),\quad t\in(0,T]\\
u_\alpha(0)&=u_\alpha^0,\quad \alpha\in\mathcal{I}.\nonumber
\end{align} 
By virtue of (B1), (B2) and (A3) the Cauchy problem \eqref{CP za LJ} fulfills all the assumptions of \cite[Theorem 5.3, p. 147]{Pazy} so there exists a unique classical solution $u_\alpha\in C^1([0,T],X)$ given by
$$u_\alpha(t)=S_\alpha(t,0)u_\alpha^0+\int_0^tS_\alpha(t,s)f_\alpha(s)ds,\quad t\in[0,T]$$
for all $\alpha\in\mathcal{I}.$

It remains to show that $u=\sum_{\alpha\in\mathcal{I}}u_\alpha H_\alpha\in C([0,T],X)\otimes (S)_{-1},$ i.e. that there exists $q>0$ such that $\sum_{\alpha\in \mathcal{I}}\sup_{t\in [0,T]}\|u_\alpha(t)\|_X^2(2\mathbb{N})^{-q\alpha}<\infty.$

Without loss of generality, we may assume that the constants $K,p>0$ are such that for all $\alpha\in \mathcal{I}$
\begin{align*}
\|u_\alpha^0\|_X&\leq K(2\mathbb{N})^{p\alpha}\\
\sup_{t\in[0,T]}\|f_\alpha(t)\|_X&\leq K(2\mathbb{N})^{p\alpha}.
\end{align*}
Now, for all $\alpha\in\mathcal{I},$ we obtain
\begin{align*}
\sup_{t\in[0,T]}\|u_{\alpha}(t)\|_X & \leq\sup_{t\in[0,T]}\left\{\|S_\alpha(t,0)\|\|u_\alpha^0\|_X+\int_0^t\|S_\alpha(t,s)\|\|f_\alpha(s)\|_Xds\right\}\\
& \leq \sup_{t\in[0,T]}\left\{\|S_\alpha(t,0)\|\|u_\alpha^0\|_X+\sup_{s\in[0,t]}\|S_\alpha(t,s)\|\|f_\alpha(s)\|_X\int_0^t ds\right\}\\
& \leq \sup_{t\in[0,T]}\left\{me^{wt}K(2\mathbb{N})^{p\alpha}+me^{wt}K(2\mathbb{N})^{p\alpha}t\right\}\\
& \leq (1+T)me^{wT}K(2\mathbb{N})^{p\alpha}.
\end{align*}
Finally, for $q>2p+1$ we obtain
\begin{align*}
\sum_{\alpha\in\mathcal{I}}\sup_{t\in[0,T]}\|u_\alpha(t)\|^2_X (2\mathbb{N})^{-q\alpha}\leq \left((1+T)me^{wT}K\right)^2\sum_{\alpha\in\mathcal{I}}(2\mathbb{N})^{(2p-q)\alpha}<\infty.
\end{align*}
\end{proof}



   
   \section{Extensions and applications}
  Our results can be extended to a far more general case of stochastic evolution equation of the form
   \begin{equation}
   \label{stoch -polynomial nonlinearity}
   \begin{split}
   u_t (t, \omega)&= \mathbf A \, u(t, \omega) + p_n^{\lozenge }(u(t, \omega))+f(t,\omega), \quad t\in (0,T]\\
   u(0,\omega) &= u^0(\omega), \qquad \omega\in\Omega,
   \end{split}
   \end{equation}
   with a Wick-polynomial type of nonlinearity 
   \begin{equation}
   p_n^{\lozenge }(u)= \sum_{k=0}^n \, \, a_k \,  u^{\lozenge k} = 
   a_0 + a_1 \, u +  a_{2} \, u^{\lozenge 2}    + a_{3} \, u^{\lozenge 3} + \dots a_n\, u^{\lozenge n} ,
   \label{Wick-polinom}
   \end{equation}where $a_n\not=0$ and $a_k$, $0\leq k \leq n$ are either constants or deterministic functions. Equation \eqref{stoch -polynomial nonlinearity} generalizes  equation \eqref{NLJ}
   and it can be solved by the very same method presented in the paper, provided that one stipulates that the corresponding deterministic version of \eqref{stoch -polynomial nonlinearity} has a solution and modifies assumption $(A4-n)$ correspondingly. Hence, we replace $(A4-n)$ with the following assumption:
   
   \begin{enumerate}
   	\item[(A4-pol-n)]  The Cauchy problem
   	\[\frac{d}{dt} u_{\mathbf{0}} (t) =  A_{\mathbf{0}} u_{\mathbf{0}}  (t) +   p_n (u_{\mathbf{0}}(t)) + f_\mathbf{0}(t) ,\quad t\in(0,T]; \quad u_{\mathbf{0}}(0) = u_{\mathbf{0}}^0,\]
   	has a classical solution $u_{\mathbf{0}}\in C^1([0,T],X)$, where  
   	 \begin{equation}
   	 p_n(u)= \sum_{k=0}^n \, \, a_k \,  u^{k} = 
   	 a_0 + a_1 \, u +  a_{2} \, u^{ 2}      + a_{3} \, u^{3} + \dots a_n\, u^{ n} ,
  \label{polinom}
   	 \end{equation}is a classical polynomial of degree $n$ corresponding to the Wick-polynomial \eqref{Wick-polinom}. 
   \end{enumerate}
   
   We extend Theorem \ref{main}, and  for the sake of technical simplicity,  present only a procedure for solving  \eqref{stoch -polynomial nonlinearity} for $n=3$, but note that the general case may be done mutatis mutandis. 
   
First we note that  from the form of the process \eqref{proces} and from the form of its Wick-powers \eqref{WP2}, as well as from \eqref{WP3} we obtain the expansion of the Wick-polynomial nonlinearity 
   \begin{equation}
   \begin{split}
  &p_3^{\lozenge }(u) =
   a_0 + a_1 \, u +  a_{2} \, u^{\lozenge 2}    + a_{3} \, u^{\lozenge 3} \\
   &= a_0 H_{\mathbf{0}} + a_1 \bigg( u_{\mathbf{0}} H_{\mathbf{0}} + \sum_{|\alpha|>0}  u_\alpha H_\alpha\bigg) + a_2 \bigg( u^2_{\mathbf{0}} H_{\mathbf{0}} +   \sum_{|\alpha|>0} \Big(2u_{\mathbf{0}} u_\alpha + \sum_{\mathbf{0}<\beta < \alpha} u_{\beta} u_{\alpha-\beta}\Big) \, H_\alpha \bigg) + \\
   & +
   a_3 \bigg( u^3_{\mathbf{0}} H_{\mathbf{0}} +  \sum_{|\alpha|>0}  \Big(3 u_{\mathbf{0}}^2 u_\alpha + 3  u_{\mathbf{0}} \sum_{0<\beta< \alpha }  u_{\alpha-\beta}  u_\beta + 
   \sum_{0<\beta < \alpha } \sum_{0<\gamma< \beta }    u_{\alpha-\beta}  u_{\beta-\gamma}   u_\gamma(t)\Big) \,  H_\alpha \bigg). 
   \label{p3 expansion}
\end{split}
   \end{equation}
   When summing up the corresponding coefficients, the expression \eqref{p3 expansion} transforms to 
   \begin{align*}
   p_3^{\lozenge }(u)&=
  (a_0 + a_1 u_{\mathbf{0}} +  a_{2} \, u^2_{\mathbf{0}}   + a_{3} \, u^3_{\mathbf{0}}) \, H_{\mathbf{0}}\\& +
  \sum_{\alpha>\mathbf{0}}\Big((3a_3 u^2_{\mathbf{0}} + 2 a_2 u_{\mathbf{0}} + a_1) \, u_\alpha + (3a_3 u_{\mathbf{0}} + a_2)  \sum_{0<\beta< \alpha }  u_{\alpha-\beta} u_\beta \\
  &+ a_3  \sum_{0<\beta < \alpha } \sum_{0<\gamma< \beta }    u_{\alpha-\beta}  u_{\beta-\gamma}   u_\gamma \Big) \, H_\alpha \\
  &= p_3(u_{\mathbf{0}}) + \sum_{\alpha> \mathbf{0}} \Big( p'_3(u_{\mathbf{0}}) u_\alpha + \frac{1}{2!} \cdot  p''_3(u_{\mathbf{0}})  \sum_{0<\beta< \alpha }  u_{\alpha-\beta} u_\beta \\
  &+ \frac1{3!} \cdot p'''_3(u_{\mathbf{0}})  \sum_{0<\beta < \alpha } \sum_{0<\gamma< \beta }    u_{\alpha-\beta}  u_{\beta-\gamma}   u_\gamma  \Big) \, H_\alpha ,
   \end{align*} where $p_3'$, $p_3''$ and $p_3'''$ denote the first, the second and the third derivative of the polynomial \eqref{polinom}, respectively. \\
   Thus, by applying the Wiener-It\^o chaos expansion method to  the nonlinear stochastic problem  \eqref{stoch -polynomial nonlinearity} we obtain 
    the system of infinitely many deterministic Cauchy problems: 
   \begin{enumerate}
   	\item[$1^\circ$] for $\alpha =\mathbf{0}$ 
   	\begin{equation}
   	\label{nulta jed-opste}
   	\frac{d}{dt} u_{\mathbf{0}} (t) =  A_{\mathbf{0}} u_{\mathbf{0}}  (t) +  p_3(u_{\mathbf{0}} (t)) +f_\mathbf{0}(t), \quad u_{\mathbf{0}}(0) = u_{\mathbf{0}}^0,  
   	\end{equation}
   	and 
   	\item[$2^\circ$] for  $\alpha >\mathbf{0}$
   	\begin{equation}
   	\label{sistem alfa -polinomna jed}
   	\begin{split}
   	\frac{d}{dt}u_\alpha (t) &=  \big( A_\alpha +  p_3'(u_{\mathbf{0}} (t)) Id \big) \,  u_\alpha(t)  +  \frac12 \, p_3''(u_{\mathbf{0}}(t))  \sum_{0<\beta< \alpha }  u_{\alpha-\beta} (t)  u_\beta(t) + \\	
   	&+ \frac16\, p_3'''(u_{\mathbf{0}}(t))   \sum_{0<\beta < \alpha } \sum_{0<\gamma< \beta }    u_{\alpha-\beta} (t) u_{\beta-\gamma} (t)    u_\gamma(t) + f_\alpha(t), \\
   	u_\alpha (0) &=  u_\alpha^0 .
   	\end{split}
   	\end{equation} with $t\in (0,T]$ and $\omega\in\Omega$. 
   \end{enumerate}
   
  We denote by 
   \begin{align*}
   B_{\alpha,p_{3}}(t) &= A_\alpha +  p_3'(u_{\mathbf{0}} (t)) Id \qquad  \text{and}  \\
   g_{\alpha,p_{3}} (t) &= \frac12 \cdot p_3''(u_{\mathbf{0}})  \sum_{0<\beta< \alpha }  u_{\alpha-\beta} (t)  u_\beta(t) \\
   &+ 
  \frac16 \cdot p_3'''(u_{\mathbf{0}})   \sum_{0<\beta < \alpha } \sum_{0<\gamma< \beta }    u_{\alpha-\beta} (t) u_{\beta-\gamma} (t) u_\gamma(t) + f_\alpha(t), \, \label{h alfa1}
   \end{align*}
   for $t\in(0,T]$ and all $\alpha > \mathbf{0}$. Hence, the problems \eqref{sistem alfa -polinomna jed} for $\alpha > \mathbf{0}$
   can be written in the form 
   \begin{equation}
   \label{sistem opste}
   \begin{split}
   \frac{d}{dt}u_\alpha (t) &=  B_{\alpha,p_3}(t)  \, u_\alpha(t) +  g_{\alpha,p_3}(t), \quad t\in(0,T] \\ 
   u_\alpha (0) &=  u_\alpha^0 .  
   \end{split}
   \end{equation}

   \begin{Theorem}\label{T4} Let the assumptions $(A1)-(A3)$ and $(A4-pol-3)$ be fulfilled. 
   	Then, there exists a unique almost classical solution $u\in C([0,T],X)\otimes (S)_{-1}$ to \eqref{stoch -polynomial nonlinearity}.
   \end{Theorem}
   \begin{proof}
   	Under the assumptions $(A1)-(A2)$ and the assumption  $(A4-pol-3)$ that \eqref{nulta jed-opste}  has a classical solution in $C^1([0,T], X)$, it  can be proven (similarly as it was done in Lemma \ref{Lema}) that for every $\alpha>\mathbf{0}$ the evolution system \eqref{sistem opste} has a unique classical solution $u_\alpha\in C^1([0,T],X)$.  Then, 
   	in order to show that $u$ is an almost classical solution to \eqref{stoch -polynomial nonlinearity},  one has to prove that $u\in C([0,T],X)\otimes (S)_{-1}$. Indeed, this can be done  
    in an analogue way as in the proof of Theorem \ref{T3}, with  	$L_{\mathbf{0}} =  \sup_{t\in[0,T]} \|u_{\mathbf{0}}(t)\| $ and $$M_3 =\max\{ \sup_{t\in [0,T]} \|p_3(u_{\mathbf{0}}(t))\|, \sup_{t\in [0,T]} \|p_3'(u_{\mathbf{0}}(t))\|, \sup_{t\in [0,T]} \|p_3''(u_{\mathbf{0}}(t))\|, \sup_{t\in [0,T]} \|p_3'''(u_{\mathbf{0}}(t))\|\}.$$   
   \end{proof}

   \subsection{Examples}
   We present two classes of stochasic reaction-diffusion equations that belong to the class of problems \eqref{stoch -polynomial nonlinearity}. 
   \subsubsection{Stochastic generalized FitzHugh-Nagumo equation}
   The nonlinear stochastic evolution equation 
   \begin{equation}
   \label{stoch FHN eq}
   \begin{split}
   u_t (t, \omega)&= \mathbf A \, u(t, \omega) + u^{\lozenge 2}(t, \omega) - u^{\lozenge 3}(t, \omega)+f(t,\omega), \quad t\in (0,T]\\
   u(0,\omega) &= u^0(\omega), \quad \omega\in\Omega,
   \end{split}
   \end{equation}which belongs to the class of generalized FitzHugh-Nagumo equations is an equation of type \eqref{stoch -polynomial nonlinearity}.  Particularly,  for $\mathbf A=\triangle$, the corresponding reaction-diffusion deterministic equation 
   \begin{equation}\label{determ FHN eq}
   u_t= \triangle u(t) + F(u(t)),  \quad u(0)=u^0, 
   \end{equation}  with a nonlinearity of the form $F(u)=-u(a-u)(b-u)$  is the celebrated FitzHugh-Nagumo equation, which arises  in  various models of neurophysiology. The equation \eqref{determ FHN eq} has been introduced by FitzHugh and Nagumo \cite{FH, Nagumo}  in order to model the conduction of electrical impulses in a nerve axon. A stochastic version of the FitzHugh-Nagumo equation \eqref{determ FHN eq} was studied in \cite{Albeverio}, while a control problem for  the FitzHugh-Nagumo equation perturbed by coloured Gaussian noise was solved in \cite{Barbu}. Clearly, the equation \eqref{stoch FHN eq} is generalizing \eqref{determ FHN eq} if we choose  $a=0$ and $b=1$ in the form of  $F(u)$. For the choice of $a=b=0$ the equation \eqref{stoch FHN eq} reduces to the Fujita type equation \eqref{NLJ}. 
   
    Here, by appying Theorem \ref{T4}, we obtain a unique almost classical solution of  the equation \eqref{stoch FHN eq}.

   \subsubsection{Stochastic generalized Fisher-KPP equation}
   The deterministic nonlinear equation of the form \eqref{determ FHN eq} with $F(u)=a u(1-u)$ is called the Fisher equation (also known as the Kolmogorov-Petrovsky-Piskunov equation). Such equations occur in  phase transition problems arising in   biology, ecology, plasma physics \cite{Fisher, KPP} etc. Particularly, such an equation  provides a deterministic model for the density of a population living in an environment with a limited carrying capacity. It also describes the wave progression of an epidemic outbreak or  the spread of an advantageous gene within a
   population. Other applications in medicine involve the modeling of cellular reactions to the introduction of toxins, voltage propagation through a nerve axon, and the process of epidermal wound healing \cite{Aronson}. In other research areas it has been also used to study flame propagation of fire outbreaks, and neutron flux in nuclear reactors. 
   
   Stochastic models that  include random effects due to some external (enviromental) noise  were studied in the framework of white noise analysis   \cite{HuangLiu}, where the authors proved the existence of the traveling wave solution. In the same setting, the stochastic KPP equation, i.e. 
   heat equations with semilinear potential and perturbation by a  multiplicative noise were considered in \cite{OksendalVZ}. Under suitable assumptions, by applying the It\^o calculus, existence of a unique strong  traveling wave solution was proven, and an implicit Feyman-Kac-like formula for the solution was presented. 
  Here we consider a generalized Wick-version of the stochastic Fisher-KPP equation 
   \begin{equation*}
   \label{stoch Fisher eq}
   \begin{split}
   u_t (t, \omega)&= \mathbf A \, u(t, \omega) + u(t, \omega) - u^{\lozenge 2}(t, \omega)+f(t,\omega), \quad t\in (0,T]\\
   u(0,\omega) &= u^0(\omega), \quad \omega\in\Omega,
   \end{split}
   \end{equation*}which can be solved by applying Theorem \ref{T4}. 
   
 \subsection{Conclusion}
 In this paper we have presented a methodology for solving stochastic evolution equations involving nonlinearities of Wick-polynomial type. However, the applications and extensions of the theory do not stop here. In place of the nonlinearity $u^{\lozenge 2}$, one might consider $u\lozenge u_x$ and with appropriate modifications solve the stochastic Burgers-type equation $u_t= u_{xx} + u\lozenge u_x +f$ or the stochastic KdV equation $u_t= u_{xxx} + u\lozenge u_x +f$, coalesced into the form $u_t=\mathbf A u+ u\lozenge u_x +f$. One can also replace the nonlinearity $u^{\lozenge n}$ by $u\lozenge |u|^{n-1}$, where the modulus of a complex-valued stochastic process is understood as $|u|=\sum_{\alpha\in\mathcal I}|u_\alpha|H_\alpha$, and find explicit solutions to the stochastic nonlinear Schr\"odinger equation  $(i\hbar)u_t=\Delta u + u\lozenge |u|^{n-1} +f$.


\section*{Acknowledgement}

The paper is supported by the following projects and grants: project 174024 financed by the Ministry of Education, Science and Technological Development of the Republic of Serbia, project 451-03-01039/2015-09/26 of the bilateral scientific cooperation between Serbia and Austria, Domus grant 4814/28/2015/HTMT provided by the Hungarian Academy of Sciences,  project 142-451-2384 of the Provincial Secretariat for Higher Education and Scientific Research and a research grant for Austrian graduates.

\end{document}